\documentclass[a4paper,12pt,reqno]{amsart}
\usepackage{amsmath,graphicx,graphics,amssymb}
\newtheorem{theo}{Theorem}

\newtheorem{prop}{Proposition}

\numberwithin{equation}{section}

\newcommand{\adj}{\operatorname{ADJ}}
\newcommand{\opp}{\operatorname{OPP}}
\newcommand{\fpp}{\operatorname{FPP}}
\newcommand{\M}{\operatorname{M}}
\newcommand{\Pf}{\operatorname{Pf}}
\newcommand{\al}{\alpha}
\newcommand{\be}{\beta}
\newcommand{\ce}{\gamma}
\newcommand{\de}{\delta}

\newcommand{\epf}{\hfill{$\square$}\medskip}

\newmuskip\pFqskip
\mathchardef\pFcomma=\mathcode`, % keep a copy of the comma

\newcommand*\pFq[5]{%
  \begingroup
  \begingroup\lccode`~=`,
    \lowercase{\endgroup\def~}{\pFcomma\mkern\pFqskip}%
  \mathcode`,=\string"8000
  {}_{#1}F_{#2}\biggl[\genfrac..{0pt}{}{#3}{#4};#5\biggr]%
  \endgroup
}

\begin{document}

\title{Lozenge tilings of hexagons with arbitrary dents}
%\title{The number of lozenge tilings of a hexagon with dents on the boundary}

\author{Mihai Ciucu and Ilse Fischer}

\thanks{The authors acknowledge support by the National Science Foundation, DMS grant 1101670 and Austrian Science Foundation FWF, START grant Y463}

\begin{abstract} Eisenk\"olbl gave a formula for the number of lozenge tilings of a hexagon on the triangular lattice with three unit triangles removed from along alternating sides. In earlier work, the first author extended this to the situation when an arbitrary set of unit triangles is removed from along alternating sides of the hexagon. In this paper we address the general case when an arbitrary set of unit triangles is removed from along the boundary of the hexagon.
\end{abstract}

\maketitle

\section{Introduction}

%* motivation (extends my previous extension of Theresia's theorem;
%application of generalized Kuo; special case when dents are on
%opposite sides looks like a counterpart of the Gelfand-Tsetlin regions
%in Cohn-Larsen-Propp)

MacMahon's classical theorem \cite{MacM} on the enumeration of plane partitions that fit in an $a\times b\times c$ box is equivalent to the fact that the number of lozenge tilings\footnote{ A lozenge is the union of two adjacent unit triangles on the triangular lattice; a lozenge tiling of a lattice region $R$ is a covering of $R$ by lozenges that has no gaps or overlaps.} of a hexagon of side lengths $a$, $b$, $c$, $a$, $b$, $c$ (in cyclic order) on the triangular lattice is equal to
\begin{equation}
\label{MacM_eq}
\prod_{i=1}^a\prod_{j=1}^b\prod_{k=1}^c \frac{i+j+k-1}{i+j+k-2}.
\end{equation}
The elegance of this result has been the source of inspiration for a large amount of research in the last four decades. The questions about MacMahon's original four symmetry classes were augmented to Stanley's program \cite{StanPP} concerning a total of ten symmetry classes, all of which turn out to be enumerated by simple product formulas. Probabilistic aspects were studied by Cohn, Larsen and Propp \cite{CLP}, Borodin, Gorin and Rains \cite{Borodin}, and Bodini, Fusy and Pivoteau \cite{Bodini}. Extensions were given by the first author in \cite{ppone} and Vuleti\'c \cite{Vuletic}. 

Eisenk\"olbl \cite{Eisen} presented a refinement which gives an explicit formula for the number of lozenge tilings of a hexagon with a dent on each of three alternating sides. The first author extended this \cite{gk} to the situation when an arbitrary set of dents is placed on the union of three alternating sides. In this paper we address the general case when an arbitrary set of unit triangles is removed from along the boundary of the hexagon.

%\newpage

\begin{figure}[h]
\label{CLP_fig}
\centerline{
\hfill
{\includegraphics[width=0.45\textwidth]{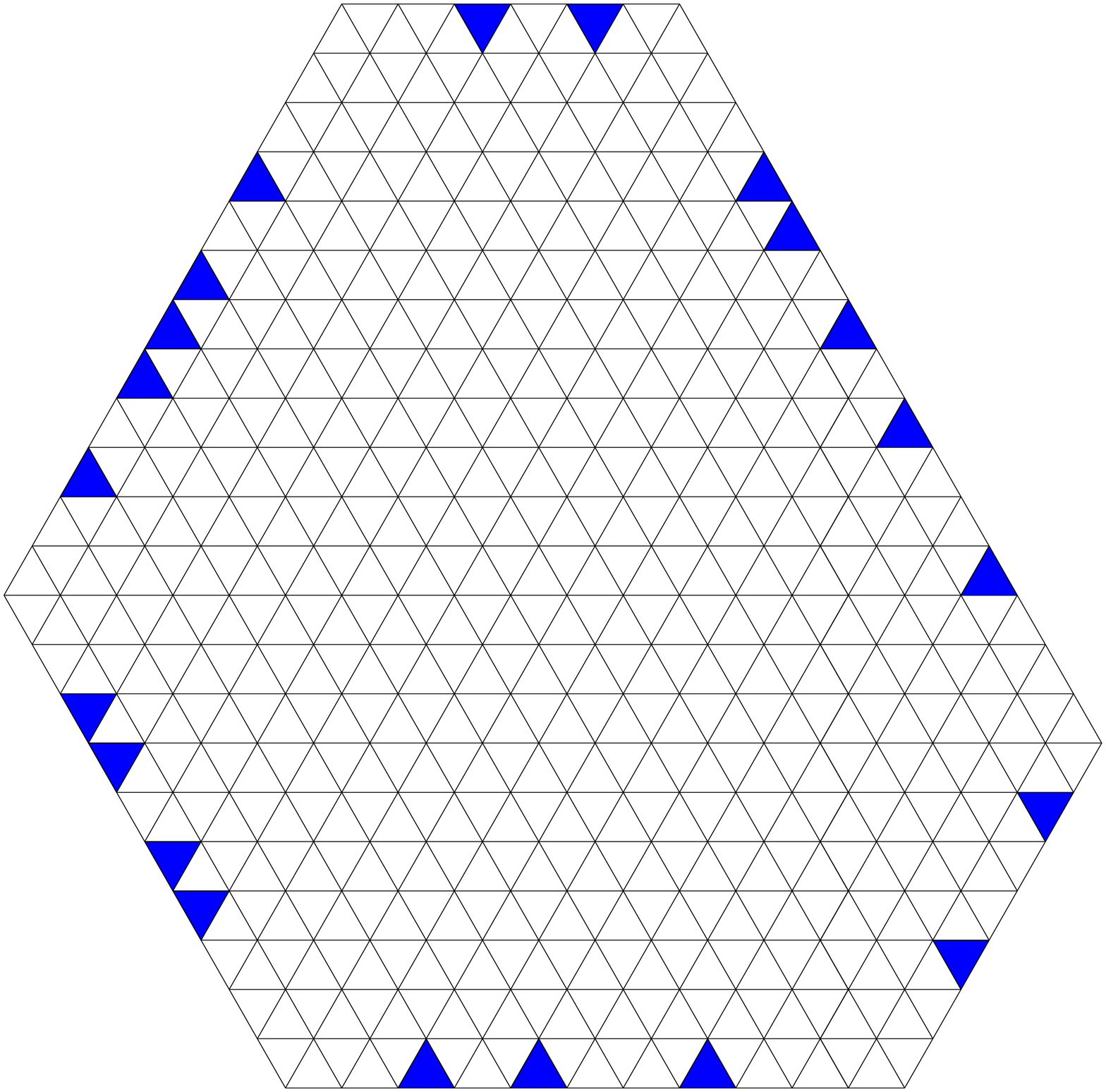}}
\hfill
{\includegraphics[width=0.45\textwidth]{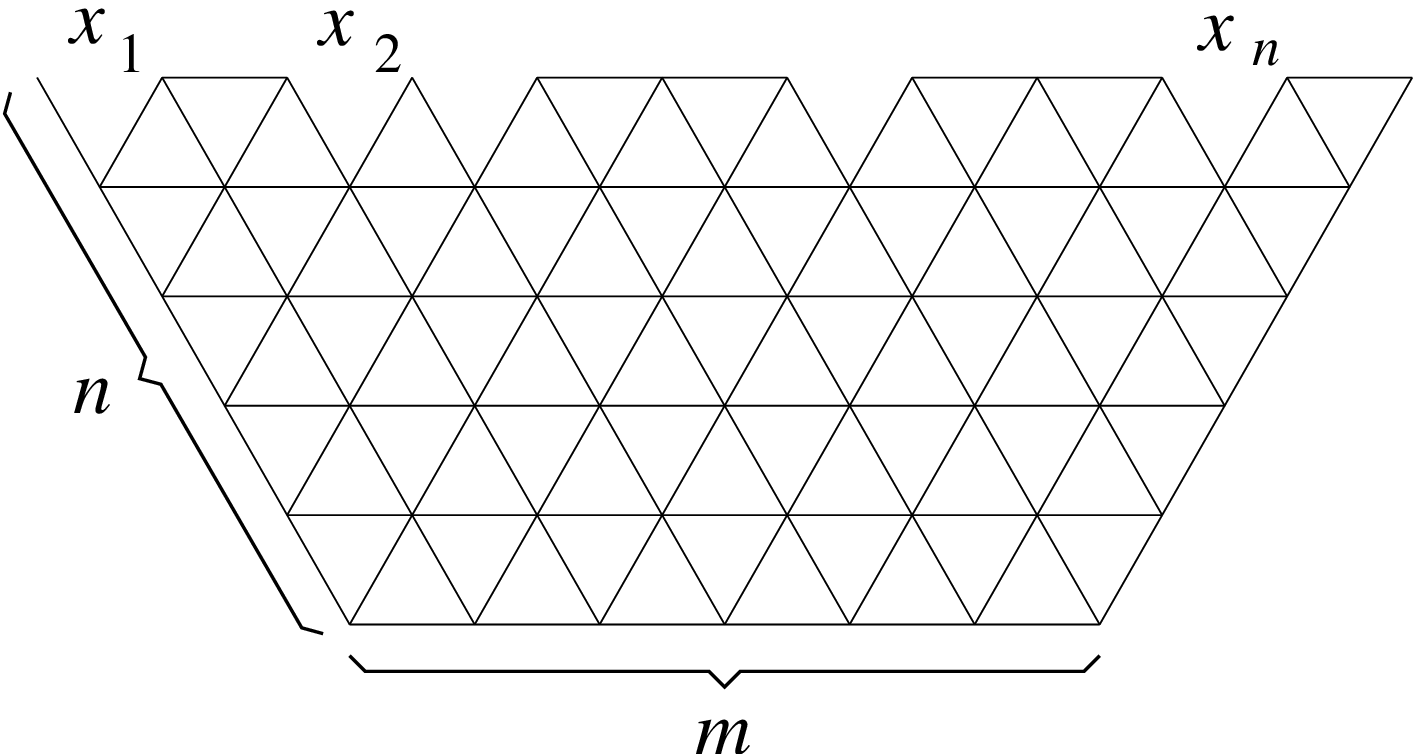}}
\hfill
}
%\vskip-0.1in
\caption{\label{gk_regions_fig1} The region $H_{6,10,7}^{5}$ with 13 up-pointing and 8 down-pointing unit triangles removed from along the boundary (left) and the region $T_{6,5}(1,3,4,7,10)$ (right).}
%\vskip-0.15in
\end{figure}

\section{Statement of main results}

%* the determinant formula (for the number of tilings of a hexagon with
%an arbitrary set of unit triangles removed from along the boundary),
%stating that all entries are given either by product formulas or by
%simple sums of round factors

%We generalize Eisenk\"olbl's regions $H_{x,y,z}^{r,s,t}$ of the previous section as follows. 

As it is easy to check, any hexagon drawn on the triangular lattice has the property that its side-lengths, listed in cyclic order, are of the form $a,b+k,c,a+k,b,c+k$, for some non-negative integers $a$, $b$, $c$ and $k$. Our regions that extend Eisenk\"olbl's result are obtained by making dents in hexagons of such side-lengths, and concern therefore the most general hexagons one can draw on the triangular lattice.

Let $H_{a,b,c}^k$ be the hexagon on the triangular lattice whose sides have lengths $a,b+k,c,a+k,b,c+k$, in clockwise order starting at the top. It is readily checked that $H_{a,b,c}^k$ has $k$ more up-pointing unit triangles than down-pointing unit triangles. Therefore, in order to create a region that can be tiled by lozenges by removing unit triangles from along the boundary, we must remove $k$ more up-pointing ones than down-pointing ones. 

There are precisely $a+b+c+3k$ up-pointing unit lattice triangles in $H_{a,b,c}^k$ that share an edge with the boundary --- $a+k$, $b+k$, {resp.} $c+k$ along the southern, northeastern, resp. northwestern sides. Choose $n+k$ of them, and denote them by $\al_1,\dotsc,\al_{n+k}$ (we will sometimes refer to them as {\it dents of type $\al$}). Choose also $n$ unit triangles from the $a+b+c$ down-pointing ones that share an edge with the boundary, and denote them by $\be_1,\dotsc,\be_k$ (we call such dents {\it dents of type $\be$}). Our extension of the regions presented in \cite{gk} (which in turn generalize Eisenk\"olbl's regions studied in \cite{Eisen}) is the family of regions of type $H_{a,b,c}^k\setminus\{\al_1,\dotsc,\al_{n+k},\be_1,\dotsc,\be_k\}$ (see Figure~\ref{CLP_fig} for an example).

%\begin{figure}
%\scalebox{0.30}{\includegraphics{most_gen_eisen.eps}}
%\caption{\label{fba} The region $H_{6,10,7}^{5}$ with 13 up-pointing and 8 down-pointing unit triangles removed from along the boundary.}
%\end{figure}

For convenience, we state below two results that are referenced in the statement of our main theorem. 

%\begin{figure}
%\scalebox{0.45}{\includegraphics{CLP.eps}}
%\caption{\label{CLP_fig} The region $T_{6,5}(1,3,4,7,10)$.}
%\end{figure}

The first of them is Cohn, Larsen and Propp's \cite{CLP} translation to lozenge tilings of a classical result of Gelfand and Tsetlin \cite{GT}.

In view of the fact that lozenge tilings of a region can be identified with perfect matchings of its planar dual, for any region $R$ on the triangular lattice we denote by $\M(R)$ the number of lozenge tilings of $R$.

\begin{prop} \cite[Proposition 2.1]{CLP}
\label{CLP}
Let $T_{m,n}(x_1,\dotsc,x_n)$ be the region obtained from the trapezoid of side lengths $m$, $n$, $m+n$, $n$ (clockwise from bottom) by removing the down-pointing unit triangles from along its top that are in positions $x_1,x_2,\dotsc,x_n$ as counted from left to right (see Figure~\ref{CLP_fig} for an illustration). Then
\begin{equation}
\label{CLP_form}
\M(T_{m,n}(x_1,\dotsc,x_n))=\prod_{1\leq i<j\leq n}\frac{x_j-x_i}{j-i}.
\end{equation}
\end{prop}  

The second is a result we quote from \cite{gk} (in the notation of Theorem~\ref{main_thm} below, this result involves two related families of regions that occur when an $\al_i$ and a $\ce_j$ are removed from a certain augmented version of the region $H_{a,b,c}^k$ --- namely, the region $\bar{H}_{a,b,c}^k$ described in the statement of Theorem \ref{main_thm}).

\begin{figure}[h]
\centerline{
\hfill
{\includegraphics[width=0.45\textwidth]{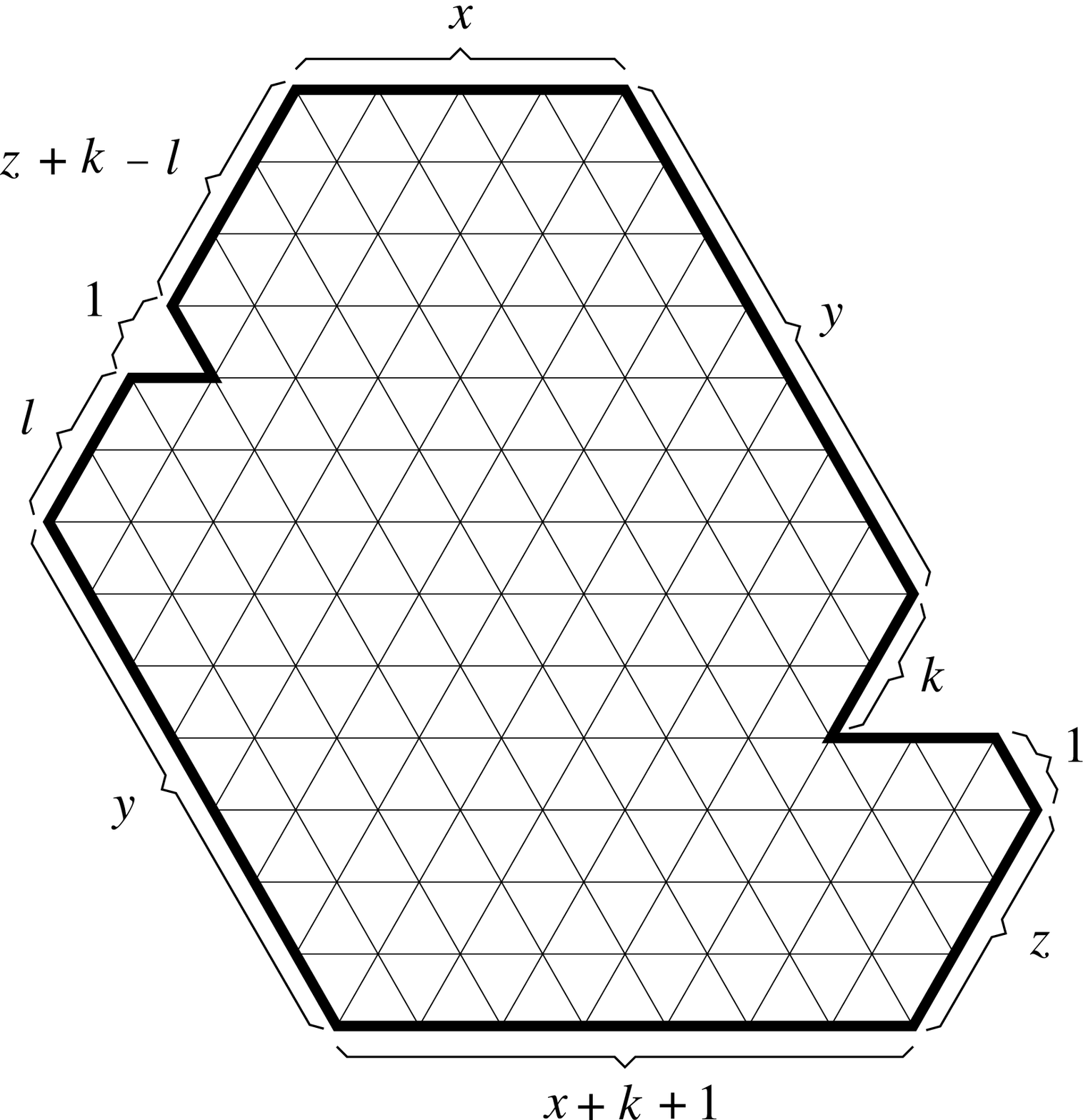}}
\hfill
{\includegraphics[width=0.45\textwidth]{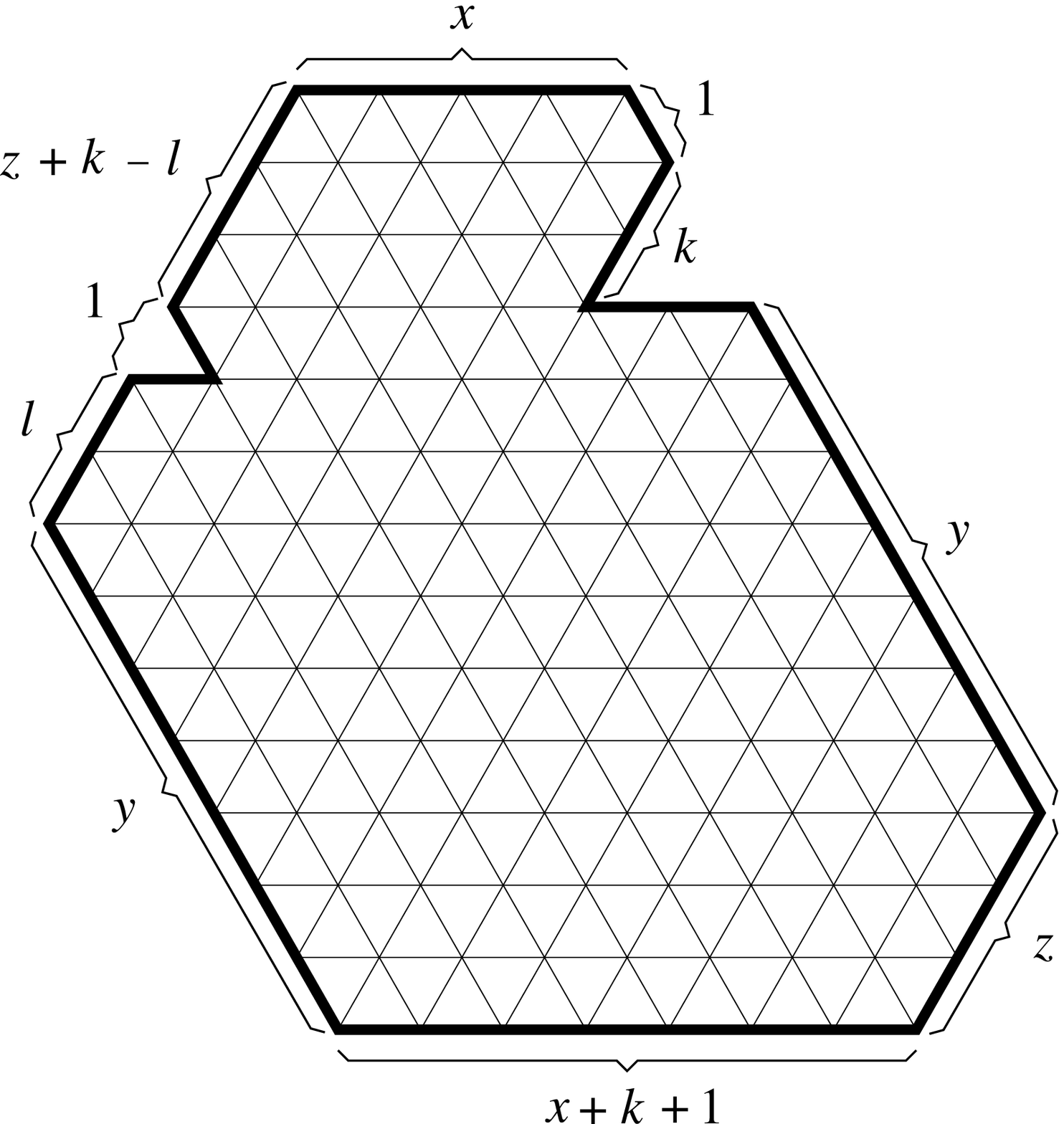}}
\hfill
}
%\vskip-0.1in
\caption{\label{gk_regions_fig2} The hexagons with two notches $H_{4,7,3}(2,2)$ (left) and $H'_{4,7,3}(2,2)$ (right).}
%\vskip-0.15in
\end{figure}

\begin{prop} \cite[Proposition 4.2]{gk}
\label{gk_regions_prop}
$(${\rm a}$)$. Let $H_{a,b,c}(k,l)$ be the region obtained from the hexagon of side lengths $a$, $b+k+1$, $c$, $a+k+1$, $b$, $c+k+1$ $($clockwise from top$)$ by removing an up-pointing unit triangle from its northwestern side, $l$ units above the western corner, and an up-pointing triangle of side $k$ from its northeastern side, one unit above the eastern corner $($see the picture on the left in Figure~\ref{gk_regions_fig2} for an illustration$)$.

Let $m=\min(a,b)$ and $M=\max(a,b)$. Then we have
\begin{equation}
\label{gk_form1a}
\M(H_{a,b,c}(k,l))=\M(H_{a,b,c})\frac{p(c,l)}{p(0,0)},
\end{equation}
where $\M(H_{a,b,c})$ is given by $(\ref{MacM_eq})$, and the polynomial $p(c,l)$ is defined to be
\begin{multline}
\label{gk_form1b}
p(c,l):=(l+1)_b (c+k-l+1)_a 
\\
%&\ \ \ \ 
\times
(c+k+2)(c+k+3)^2\cdots(c+k+m+1)^m (c+k+m+2)^m\cdots(c+k+M+1)^m 
\\
%&\ \ \ \ 
\times
(c+k+M+2)^{m-1} (c+k+M+3)^{m-2}\cdots (c+k+M+m)
\\
%&\ \ \ \ 
\times 
\sum_{i=1}^{k+1} \frac{(-1)^{i-1}}{(i-1)!(k-i+1)!} (l-k+i)_{k-i+1} (l+b+1)_{i-1} (c+1)_{i-1} (c+i+1)_{k-i+1}.
\end{multline}

$(${\rm b}$)$. Let $H'_{a,b,c}(k,l)$ be the region defined precisely as $H_{a,b,c}(k,l)$, with the one exception that the up-pointing triangle of side $k$ is one unit below the northeastern corner, rather than one unit above the eastern corner $($see the picture on the right in Figure~\ref{gk_regions_fig2} for an illustration$)$.

Let $\nu=\min(b-1,k)$, and define $r(c)$ by
\begin{equation}
\label{dz_form}
r(c):= \begin{cases} 
(c+2)^1\cdots(c+\nu+1)^{\nu}\cdots(c+b+k-\nu)^{\nu}\cdots(c+b+k-1)^1,  &\nu\geq1\\
\ \ \ \ \,1,  &\nu=0\\
\dfrac{1}{(c+1)_k},  &\nu=-1
\end{cases}
\end{equation}
%
% n/2 &\mbox{if } n \equiv 0 \\
%(3n +1)/2 & \mbox{if } n \equiv 1. \end{cases} \pmod{2}
%
$($in the first branch the bases are incremented by 1 from each factor to the next; the exponents are incremented by one until they reach $\nu$, stay equal to $\nu$ across the middle portion, and then they decrease by one unit from each factor to the next$)$.

Then we have
\begin{equation}
\label{gk_form2a}
\M(H'_{a,b,c}(k,l))={a+k\choose k}\frac{q(c,l)}{q(0,0)},
\end{equation}
where the polynomial $q(c,l)$ is defined to be
\begin{multline}
\label{gk_form2b}
%&
q(c,l):=r(c)\,(l+1)_b (z+k-l+1)_a 
\\
%&\ \ \ \ 
\times
(c+k+2)(c+k+3)^2\cdots(c+k+m+1)^m (c+k+m+2)^m\cdots(c+k+M+1)^m 
\\
%&\ \ \ \ 
\times
(c+k+M+2)^{m-1} (c+k+M+3)^{m-2}\cdots (c+k+M+m)
\\
%&\, 
\times 
\sum_{i=1}^{k+1} \frac{(-1)^{i-1}}{(i-1)!(k-i+1)!} (l-k+i)_{k-i+1} (l+b+1)_{i-1} (l-k-c)_{i-1} (l-k-c+i)_{k-i+1}
\end{multline}
$($as in part \text{\rm (a)}, $m=\min(a,b)$ and $M=\max(a,b)$$)$.

\end{prop}

%TO DO: 
%
%1. Explain cyclic order when have pending edges
%
%2. Recall previous regions we need? (MacM; CLP; two kinds from [gk])

We are now ready to state the three main results of this paper. The first one concerns the case when the dents are confined to five of the six sides of the hexagon, and provides a Pfaffian expression for the number of tilings, with each entry in the Pfaffian being given explicitly either by a simple product of linear factors, or by a single sum of products of linear factors. The second covers the general case (dents are allowed to be anywhere along the six sides of the hexagon), and provides a nested Pfaffian expression, in which 
%in addition to explicit entries as above, some of 
the entries are in their turn Pfaffians, namely of the type described above.

Define $\bar{H}_{a,b,c}^k$ to be the region obtained from $H_{a,b,c}^k$ by augmenting it with one string of~$k$ contiguous down-pointing unit triangles along its bottom as shown on the left in Figure~\ref{fbb}. Denote the $k$ down-pointing unit triangles in this string by $\ce_1,\dotsc,\ce_k$.

For a skew-symmetric matrix $A=(a_{ij})_{1\leq i,j\leq 2k}$, it will be convenient to denote its Pfaffian by $\Pf[(a_{ij})_{1\leq i<j\leq 2k}]$.

\begin{theo} %[arbitrary dents on five sides of the hexagon]{}
\label{main_thm}
Assume that one of the three sides on which dents of type $\be$ can occur does not actually have any dents on it. Without loss of generality, suppose this is the southwestern side.
 %, starting with the leftmost on the southern side and going counterclockwise. 
Let $\de_1,\dotsc,\de_{2n+2k}$ be the elements of the set $\{\al_1,\dotsc,\al_{n+k}\}\cup\{\be_1,\dotsc,\be_n\}\cup\{\ce_1,\dotsc,\ce_k\}$ listed in a cyclic order.\footnote{ If ties occur --- i.e., two of these unit triangles are encountered at the same time as one moves around the boundary of the hexagon --- they can be broken arbitrarily, and we call cyclic any of the resulting orders. If $\al_1$ (resp., $\ce_1$) is the leftmost $\al_i$ (resp., $\ce_i$) along the bottom side in the picture on the right in Figure~\ref{fbb}, $\be_1$ is the bottommost $\be_i$ along the southeastern side, and $\al_1,\dotsc,\al_{13}$ (resp., $\be_1,\dotsc,\be_8$, and $\ce_1,\dotsc,\ce_5$) occur in counterclockwise order, then one such cyclic order of the union of the $\al_i$'s, $\be_i$'s  and $\ce_i$'s is for instance $\ce_1,\ce_2,\al_1,\ce_3,\ce_4,\al_2,\ce_5,\al_3,\be_1,\be_2,\be_3,\be_4,\al_4,\al_5,\al_6,\al_7,\al_8,\be_5,\be_6,\be_7,\be_8,\al_9,\al_{10},\al_{11},\al_{12},\al_{13}$.}
%$b_1,b_2,a_1,b_4,a_2,a_3,a_4,b_5,a_5,b_6,b_7,a_6,a_7$??}.

\begin{figure}[h]
\centerline{
\hfill
{\includegraphics[width=0.45\textwidth]{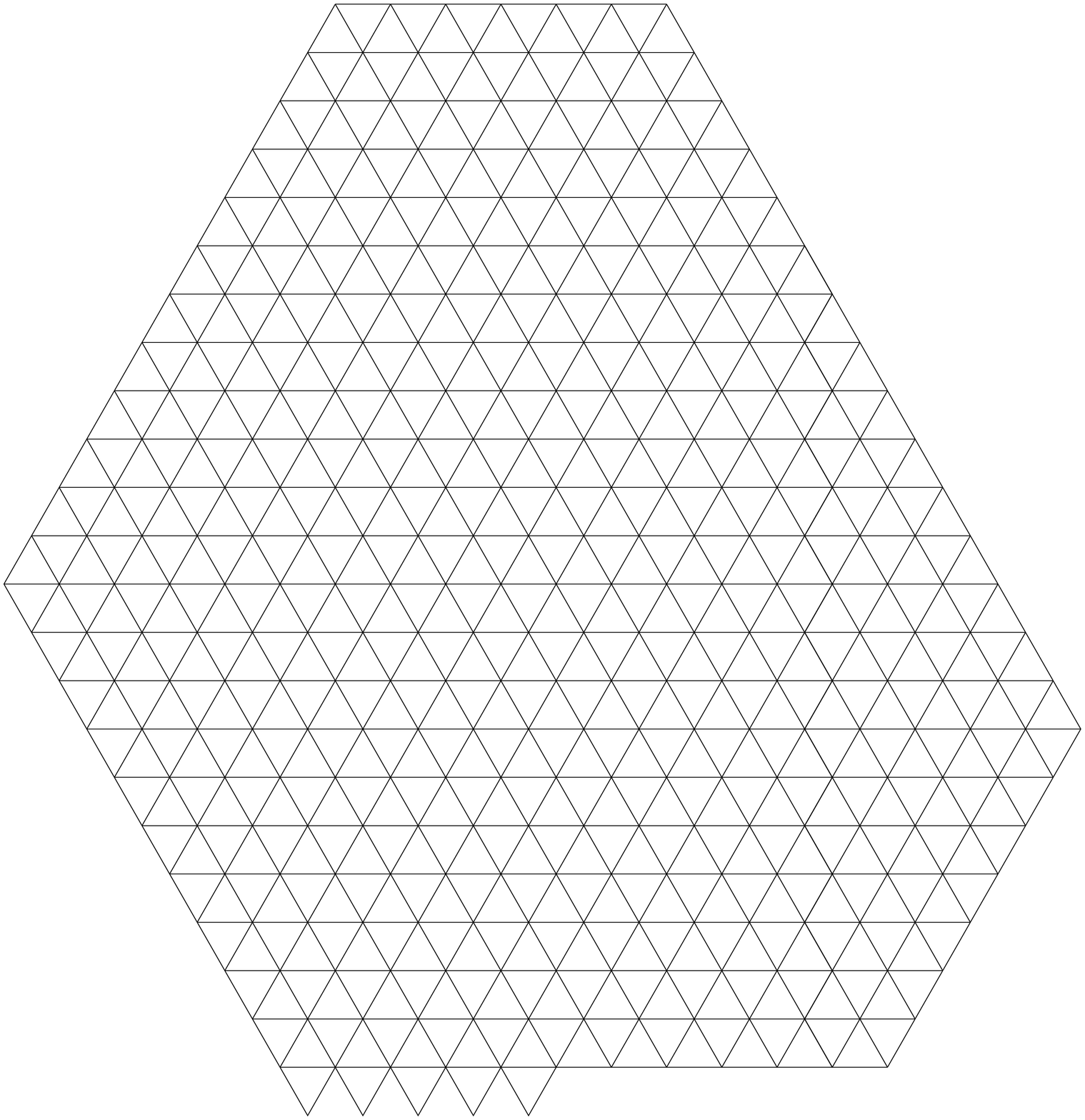}}
\hfill
{\includegraphics[width=0.45\textwidth]{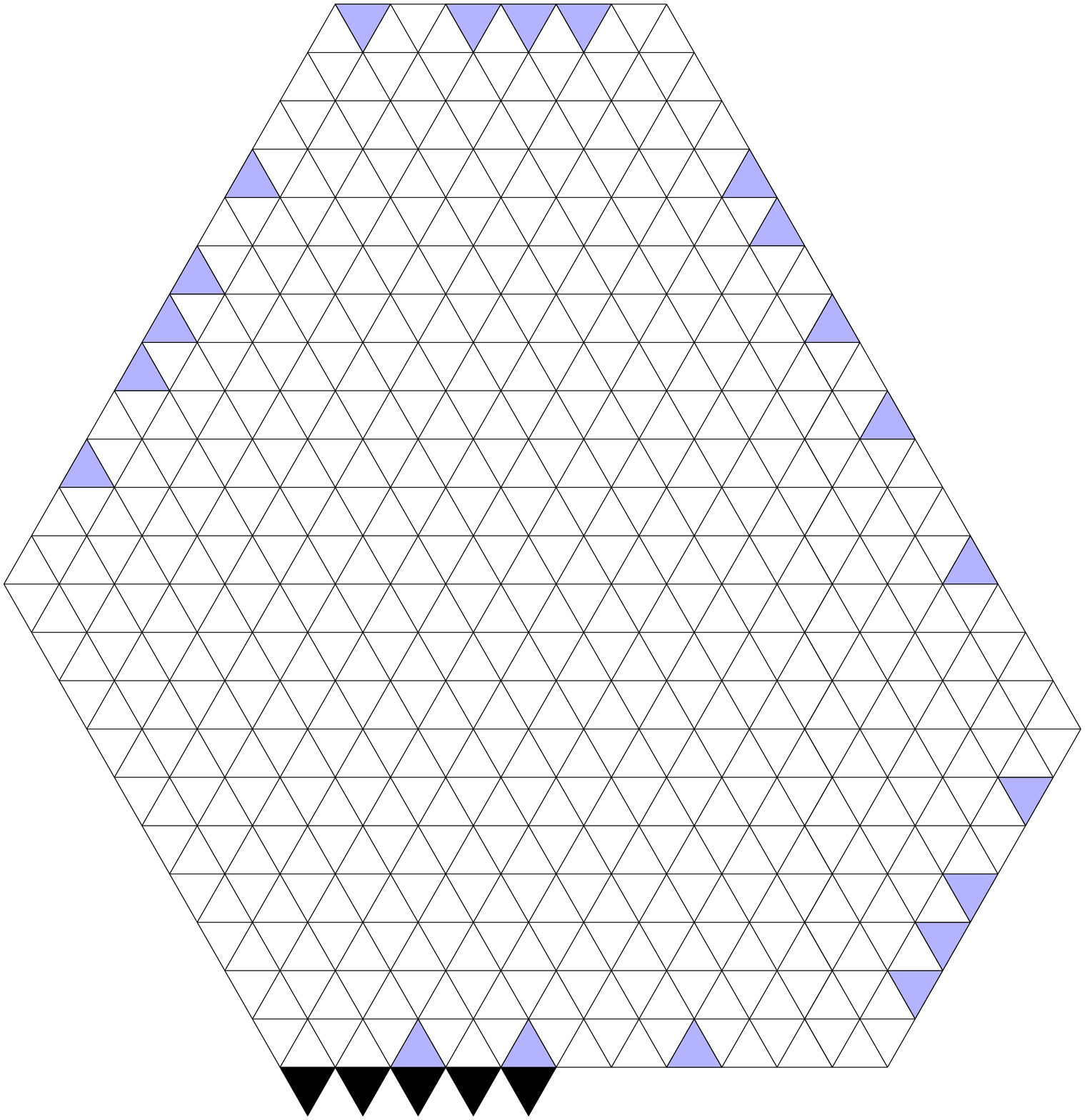}}
\hfill
%{\includegraphics[width=0.195\textwidth]{two4fs.eps}}
%\hfill
}
%\vskip-0.1in
\caption{\label{fbb} The region $\bar{H}_{6,10,7}^5$ and a choice of unit triangles along five of its sides.}
%\vskip-0.15in
\end{figure}

Then we have
\begin{equation}
\label{pfaff_expr}
\M(H_{a,b,c}^k\setminus\{\al_1,\dotsc,\al_{n+k},\be_1,\dotsc,\be_n\}) 
=
\frac{1}{\left[\M(\bar{H}_{a,b,c}^k)\right]^{n+k-1}}
\Pf\left[\left(
\M(\bar{H}_{a,b,c}^k\setminus\{\de_i,\de_j\})
\right)_{1\leq i<j\leq 2n+2k}\right],
\end{equation}
%!! Careful here --- best to use Pfaffian, and list vertices in the ``cyclic'' order clarified in Remark 1. !!
where all the quantities on the right hand side are given by explicit formulas: 

$(i)$ $\M(\bar{H}_{a,b,c}^k)$ by equation~(\ref{MacM_eq}),

$(ii)$ $\M(\bar{H}_{a,b,c}^k\setminus\{\al_i,\be_j\})$ is 0 if $\al_i$ shares an edge with one of the $\ce_\nu$'s, or if $\al_i$ is on the northwestern side, at distance\footnote{ The distance from a dent to a corner is meant in the ``infimum'' sense; e.g., on the right in Figure~\ref{fbb}, the distance between the bottommost dent on the northwestern side and the western corner is 2.} at most $k-1$ from the western corner;
otherwise, it is given by Proposition~\ref{adjacent_prop} if $\al_i$ and $\be_j$ are along adjacent sides, and by Proposition~\ref{opposite_prop} if $\al_i$ and $\be_j$ are along opposite sides,

$(iii)$ $\M(\bar{H}_{a,b,c}^k\setminus\{\al_i,\ce_j\})$ is 0 if $\al_i$ shares an edge with one of the $\ce_\nu$'s with $\nu\neq j$, or if $\al_i$ is on the northwestern side, at distance at most $j-2$ from the western corner;
otherwise is given by Proposition~\ref{CLP} if $\al_i$ and $\ce_j$ are along the same side, and by Proposition~\ref{gk_regions_prop} if $\al_i$ and $\ce_j$ are along different sides,

$(iv)$ $\M(\bar{H}_{a,b,c}^k\setminus\{\al_i,\al_j\})=\M(\bar{H}_{a,b,c}^k\setminus\{\be_i,\be_j\})=\M(\bar{H}_{a,b,c}^k\setminus\{\be_i,\ce_j\})=\M(\bar{H}_{a,b,c}^k\setminus\{\ce_i,\ce_j\})=0$.

%$\M(\bar{H}_{x,y,z}^k)$ by ..., $\M(\bar{H}_{x,y,z}^k\setminus\{a_i,b_j\})$ (resp., $\M(\bar{H}_{x,y,z}^k\setminus\{a_i,c_j\})$) by ... if $a_i$ and $b_j$ (resp., $a_i$ and $c_j$) are along the same side and by Proposition~\ref{opposite_prop}  if $a_i$ and $b_j$ (resp., $a_i$ and $c_j$) are along different sides, and $\M(\bar{H}_{x,y,z}^k\setminus\{a_i,a_j\})=\M(\bar{H}_{x,y,z}^k\setminus\{b_i,b_j\})=\M(\bar{H}_{x,y,z}^k\setminus\{b_i,c_j\})=\M(\bar{H}_{x,y,z}^k\setminus\{c_i,c_j\})=0$.
\end{theo}

%\begin{figure}[ht]
%\begin{minipage}[b]{0.55\linewidth}
%\centering
%\scalebox{0.6}{\includegraphics[width=\textwidth]{big_reg.eps}}
%\end{minipage}
%\begin{minipage}[b]{0.55\linewidth}
%\centering
%\scalebox{0.6}{\includegraphics[width=\textwidth]{big_reg_mons.eps}}
%\end{minipage}
%\caption{\label{fbb} The region $H_{6,10,7;5}^\star$ and a choice of unit triangles along its boundary.}
%\end{figure}

\begin{theo} %[arbitrary dents on all six sides of the hexagon]{}
\label{main_thm_gen}
Let $\al_1,\dotsc,\al_{n++k}$ be arbitrary dents of type $\al$ and $\be_1,\dotsc,\be_n$ arbitrary dents of type $\be$ along the boundary of ${H}_{a,b,c}^k$. Then $\M({H}_{a,b,c}^k\setminus\{\al_1,\dotsc,\al_{n+k},\be_1,\dotsc,\be_n\})$ is equal to the Pfaffian of a $2n\times2n$ matrix whose entries are Pfaffians of $(2k+2)\times(2k+2)$ matrices of the type in the statement of Theorem~\ref{main_thm}.

\end{theo}

In the special situation when the number of dents of the two types is the same (i.e., $k=0$), we can express the number of tilings as a Pfaffian with entries given by explicit formulas. Write for simplicity $H_{a,b,c}$ for ${H}_{a,b,c}^0$. 

\begin{theo} %[arbitrary dents on all six sides of the hexagon]{}
\label{kis0_thm}
Let $\al_1,\dotsc,\al_{n}$ be arbitrary dents of type $\al$ and $\be_1,\dotsc,\be_n$ arbitrary dents of type $\be$ along the boundary of ${H}_{a,b,c}$, and let $\de_1,\dotsc,\de_{2n}$ be a cyclic listing of the elements of $\{\al_1,\dotsc,\al_n\}\cup\{\be_1,\dotsc,\be_n\}$. Then
\begin{equation}
\label{kis0_expr}
\M(H_{a,b,c}\setminus\{\al_1,\dotsc,\al_{n},\be_1,\dotsc,\be_n\}) 
=
\frac{1}{\left[\M({H}_{a,b,c})\right]^{n-1}}
\Pf\left[\left(
\M({H}_{a,b,c}\setminus\{\de_i,\de_j\})
\right)_{1\leq i<j\leq 2n}\right],
\end{equation}
where the values of $\M({H}_{a,b,c}\setminus\{\de_i,\de_j\})$ are given explicitly as follows: $\M({H}_{a,b,c}\setminus\{\al_i,\be_j\})$ by Proposition~\ref{adjacent_prop} if $\al_i$ and $\be_j$ are on adjacent sides and by Proposition~\ref{opposite_prop} if $\al_i$ and $\be_j$ are on opposite sides, and $\M({H}_{a,b,c}\setminus\{\al_i,\al_j\})=\M({H}_{a,b,c}\setminus\{\be_i,\be_j\})=0$.

\end{theo}

%\newpage

\section{Two families of regions with two dents}

The formulas in this section involve hypergeometric series. Recall that the hypergeometric series of parameters $a_1,\ldots,a_r$ and $b_1,\ldots,b_s$ is defined as 
$$
\pFq{r}{s}{a_1,\ldots,a_r}{b_1,\ldots,b_s}{z} = \sum_{k=1}^{\infty} 
\frac{(a_1)_k \cdots (a_r)_k}{(b_1)_k \cdots (b_s)_k} \frac{z^k}{k!}.
$$

\begin{prop}
\label{adjacent_prop}
Let $a,b,c,j,k$ be non-negative integers with $1 \le j \le a$ and $1 \le k \le c$.
The number of lozenge tilings of the hexagon $H_{a,b,c}$ with two dents on adjacent sides of length $a$ and $c$ in positions $j$ and $k$, respectively, as counted from the common vertex of the two sides (see Figure~\ref{adjacent_fig}, left) is
$$
\prod_{i=0}^{a-1} \frac{(c+i)_{b}}{(1+i)_{b}} 
\pFq{3}{2}{-a+j,b,-c+k}{1-a-c, 1+b}{1}  
\frac{(1+b)_{a-j} (j)_{k-1} (1+c-k)_{k-1}}{(1)_{a-j} (1)_{k-1} (1+b+c-k)_{k-1}}.
$$
\end{prop}

\begin{figure}[ht]
\begin{minipage}[b]{0.49\linewidth}
\centering
\scalebox{0.6}{\includegraphics[width=\textwidth]{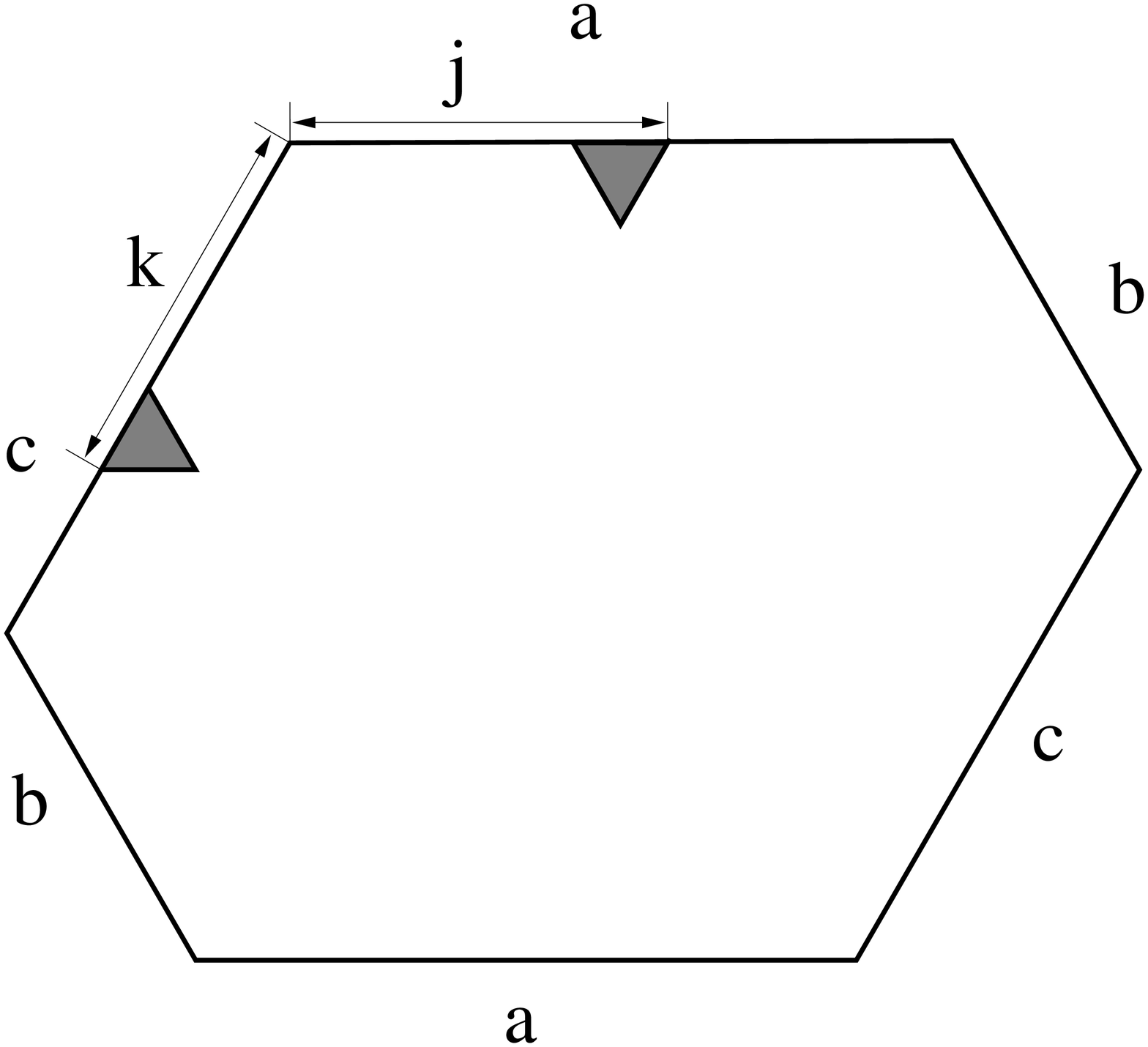}}
\end{minipage}
\begin{minipage}[b]{0.49\linewidth}
\centering
\scalebox{0.6}{\includegraphics[width=\textwidth]{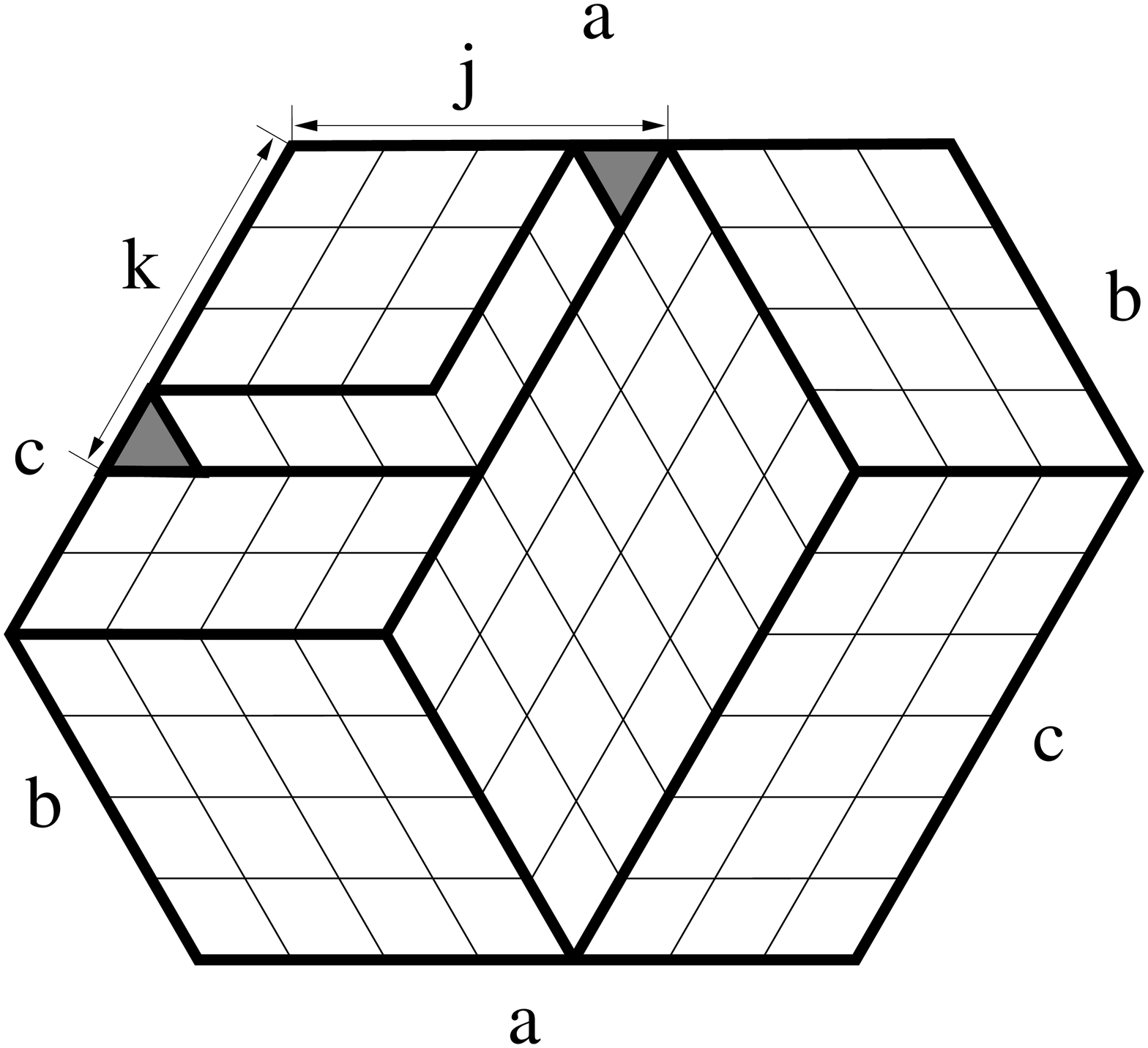}}
\end{minipage}
\caption{\label{adjacent_fig} The region in Proposition~\ref{adjacent_prop} and a canonical lozenge tiling.}
\end{figure}

The main ingredient for the proof of the proposition is the following 
theorem of Kuo. As a matter of fact, we will see later in Section~\ref{proofmainresult} that the proof of the main result of this paper is also based on the first author's generalization \cite{gk} of Kuo's result.

\begin{theo} \cite[Theorem 2.1]{Kuo}
\label{kuo1}
Let $G=(V_1,V_2,E)$ be a plane bipartite graph and $w,x,y,z$ vertices of $G$ that appear in cyclic order on a face of $G$. If $w,y \in V_1$ and $x,z \in V_2$ then
$$
\M(G) \M(G - \{w,x,y,z\}) =
\M(G - \{w,x\}) \M(G - \{y,z\}) + \M(G - \{w,z\}) 
\M(G - \{x,y\}).
$$
\end{theo}

Another important tool are various \emph{contiguous relations} for hypergeometric series. We 
found the systematic list provided in Krattenthaler's documentation\footnote{The documentation can be downloaded from {\tt http://www.mat.univie.ac.at/\~\,kratt/hyp\_hypq/hypm.pdf}.} for the computer package HYP \cite{kratt} helpful and use the notation that was introduced there. The list is based on identities given in \cite{gaspar}. Krattenthaler also proves the identities in an unpublished manuscript \cite{kratt1}. 

The concrete list of contiguous relations needed in the proof of Proposition~\ref{adjacent_prop} is the following. In these relations,
 $(A)$ and $(B)$ stand for  lists $A_1,A_2,A_3,\ldots$ and $B_1,B_2,B_3,\ldots$ of the appropriate lengths. \begin{align*}
\pFq{r}{s}{x,(A)}{y,(B)}{z} &\stackrel{{\tt C40}[x,y]}{\rightarrow} \pFq{r}{s}{x-1,(A)}{y-1,(B)}{z} 
+ \frac{(y-x) z}{(y-1)y} \frac{\prod\limits_{i=1}^{r-1} A_i }{\prod\limits_{i=1}^{s-1} B_i} \pFq{r}{s}{x,(A+1)}{y+1,(B+1)}{z} \\
\pFq{r}{s}{x,(A)}{y,(B)}{z} &\stackrel{{\tt C42}[x,y]}{\rightarrow} \frac{(y-2)(y-1)}{(y-x-1) z} \frac{\prod\limits_{i=1}^{s-1} (B_i - 1)}{\prod\limits_{i=1}^{r-1} (A_i-1)} \pFq{r}{s}{x,(A-1)}{y-1,(B-1)}{z} \\
& \qquad \qquad - \frac{(y-2)(y-1)}{(y-x-1) z} \frac{\prod\limits_{i=1}^{s-1} (B_i - 1)}{\prod\limits_{i=1}^{r-1} (A_i-1)} \pFq{r}{s}{x-1,(A-1)}{y-2,(B-1)}{z} \\
 \pFq{r}{s}{w,x,(A)}{y,(B)}{z} &\stackrel{{\tt C54}[w,x,y]}{\rightarrow} \frac{x(y-w)}{(x-w)y} \pFq{r}{s}{w,x+1,(A)}{y+1,(B)}{z} \\
& \qquad \qquad + \frac{w(y-x)}{(w-x) y} \pFq{r}{s}{w+1,x,(A)}{y+1,(B)}{z} \\
\pFq{r}{s}{w,x,(A)}{y,(B)}{z} &\stackrel{{\tt C55}[w,x,y]}{\rightarrow} \frac{(1-w+x)(y-1)}{(w-1)(1+x-y)} \pFq{r}{s}{w-1,x,(A)}{y-1,(B)}{z} \\
\\ & \qquad \qquad + \frac{x(y-w)}{(w-1)(y-x-1)} \pFq{r}{s}{w-1,x+1,(A)}{y,(B)}{z} 
\end{align*}
We shall also apply the Chu-Vandermonde summation which reads in hypergeometric notation as 
\begin{equation}
\label{chuvandermonde}
\pFq{2}{1}{x,-n}{y}{1} \stackrel{{\tt S2101}}{\rightarrow} \frac{(y-x)_n}{(y)_n},
\end{equation}
where $n$ is a non-negative integer.

\begin{proof}[Proof of Proposition~\ref{adjacent_prop}]
We prove the proposition by induction with respect to $a+b+c$. The base case of the induction follows if we show the formula for $a=1$, for $b=0$ and for $c=1$. For our argument we also need to check the cases $a=j$ and $c=k$ individually. However, the cases $a=1$ and $c=1$ follow from the cases $a=j$ and $c=k$, respectively.

\begin{figure}[ht]
\begin{minipage}[b]{0.49\linewidth}
\centering
\scalebox{0.8}{\includegraphics[width=\textwidth]{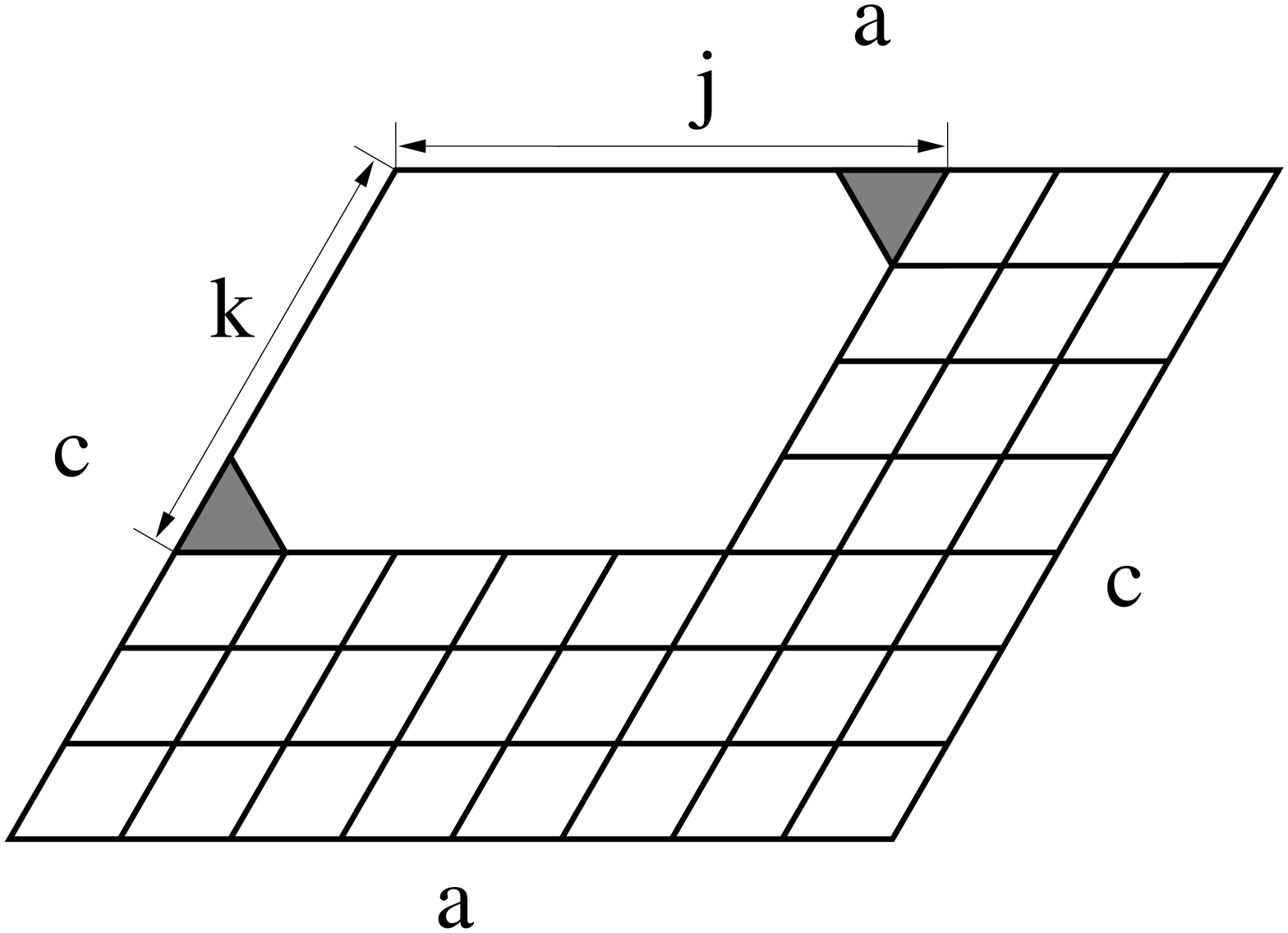}}
\end{minipage}
\begin{minipage}[b]{0.49\linewidth}
\centering
\scalebox{0.7}{\includegraphics[width=\textwidth]{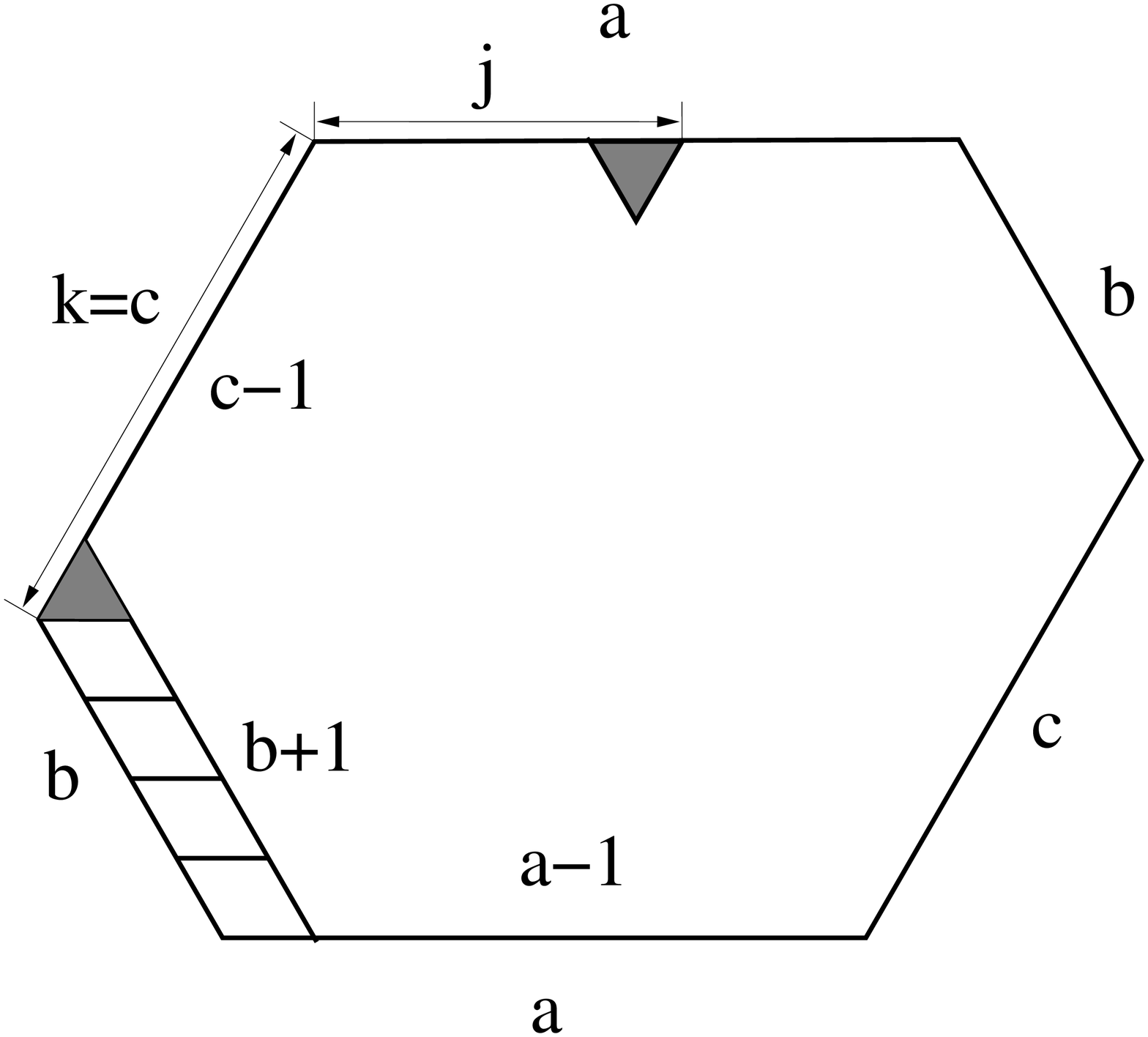}}
\end{minipage}
\caption{\label{a1_fig} The cases $b=0$ and $c=k$ in Proposition~\ref{adjacent_prop}.}
\end{figure}

\emph{Case $b=0$:}
The number is equal to the number of lozenge tilings of a hexagon with side 
lengths $j-1, 1, k-1, j-1, 1, k-1$, see Figure~\ref{a1_fig} left. The result follows again from \eqref{MacM_eq}.

\emph{Case $c=k$:}
The number is equal to the number of lozenge tilings of a hexagon with side 
lengths $a, b,c,a-1,b+1,c-1$ with a dent in position $j$ on the side of length $a$ as counted from the common vertex of this side with the side of length $c-1$, see Figure~\ref{a1_fig} right. This is a special case of Proposition~\ref{gk_regions_prop} (set $k=0$ there) or of Proposition~\ref{CLP}.

The case $a=j$ is symmetric to the case $c=k$.

\begin{figure}
\scalebox{0.25}{\includegraphics{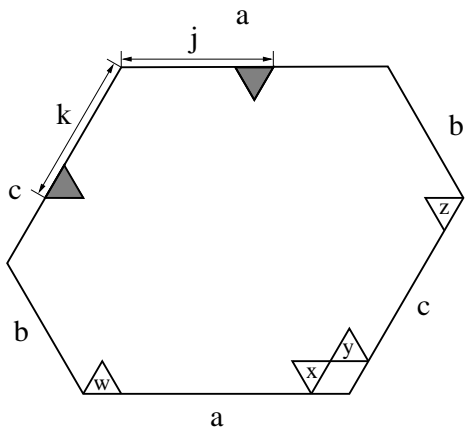}}
\caption{\label{adjacentkuo_fig} The special vertices $w,x,y,z$ in Kuo's condensation in Proposition~\ref{adjacent_prop}.}
\end{figure}

From now on we assume $a,c \ge 2$, $b \ge 1$, $j < a$ and $k < c$, and let $\adj(a,b,c)_{j,k}$ denote the number of lozenge tilings of the region. It is a well-known fact that lozenge tilings correspond to matchings of hexagonal grids and so we may use Kuo's condensation to derive a recursion for $\adj(a,b,c)_{j,k}$: we choose $w,x,y,z$ as indicated in Figure~\ref{adjacentkuo_fig}. We need to interpret the six expressions in the identity in Theorem~\ref{kuo1} in our special setting; 
$\M(G)$ counts of course all lozenge tilings of the region. In the other five cases it turns out that -- after deleting forced lozenges -- the respective quantity counts lozenge tilings of a region of the same type with changed parameters. For instance, if we delete all four triangles $w, x,y,z$, then the hexagon has now side lengths $a-1,b,c-1,a-1,b,c-1$, while the positions of the two dents is still $j$ and $k$ along adjacent the sides of lengths $a-1$ and $c-1$, respectively. 
Graphical explanations are provided in Figures~\ref{adjacentkuo12_fig}, \ref{adjacentkuo34_fig} and \ref{adjacentkuo5_fig}. In total, we obtain the following identity. 
\begin{multline}
\label{rec}
\adj(a,b,c)_{j,k} \adj(a-1,b,c-1)_{j,k} \\ = \adj(a,b,c-1)_{j,k} \adj(a-1,b,c)_{j,k} + \adj(a-1,b+1,c-1)_{j,k} \adj(a,b-1,c)_{j,k} 
\end{multline}

In Figure~~\ref{adjacent_fig} we have indicated how to construct a canonical lozenge tiling of the region under consideration  for any choice of non-negative integers $a,b,c,j,k$ with $1 \le j \le a$ and $1 \le k \le c$.
This implies that $\adj(a-1,b,c-1)_{j,k}$ is non-zero since we assume $1 \le j \le a-1$ and $1 \le k \le c-1$. This allows us to divide
\eqref{rec} by $\adj(a-1,b,c-1)_{j,k}$ and provides the recursion for $\adj(a,b,c)_{j,k}$. Indeed, the halved sum of the side lengths of the hexagonal regions of the enumerative quantities on the right-hand side are then all strictly less than $a+b+c$. 

\begin{figure}[ht]
\begin{minipage}[b]{0.49\linewidth}
\centering
\scalebox{0.6}{\includegraphics[width=\textwidth]{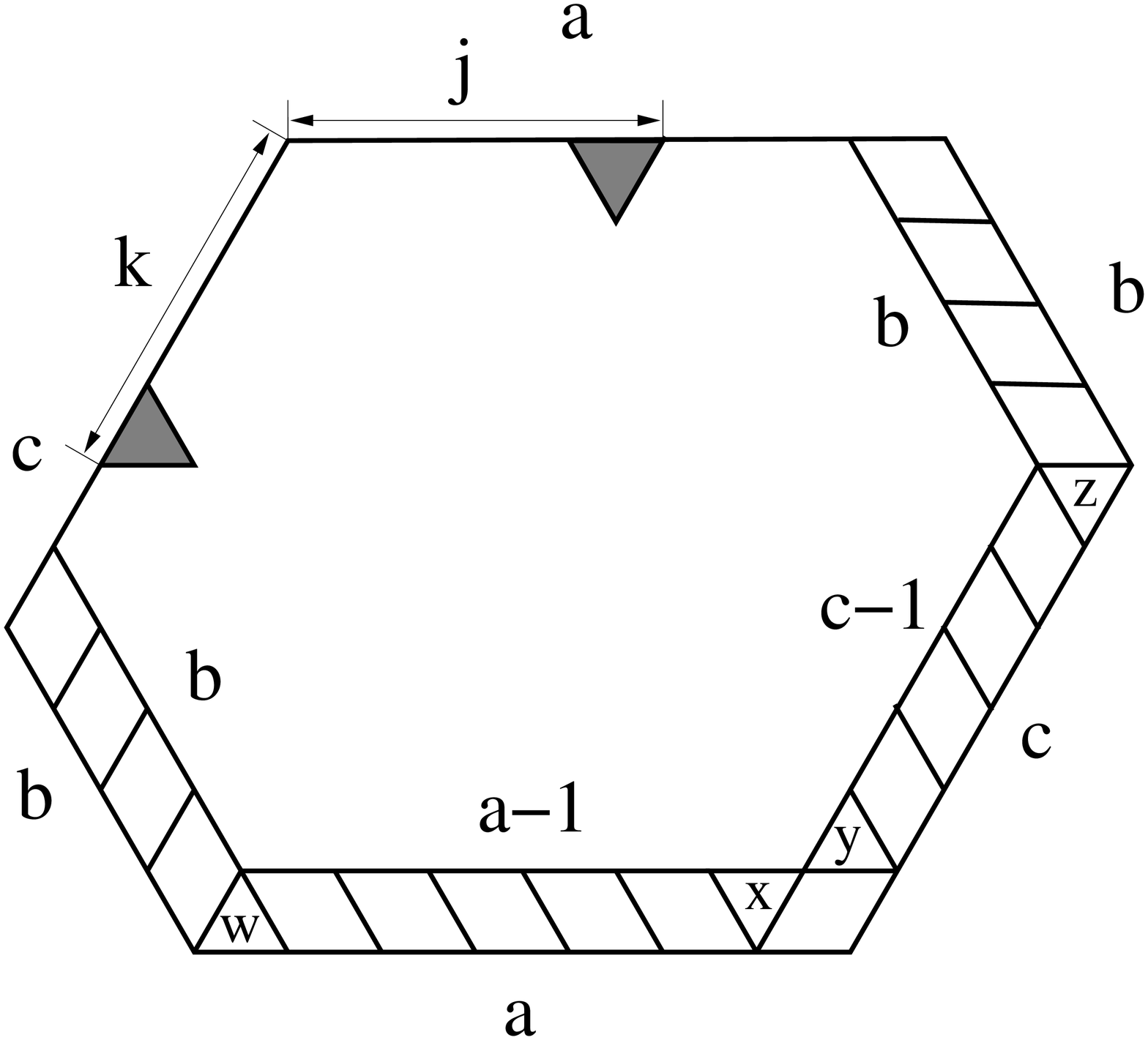}}
\end{minipage}
\begin{minipage}[b]{0.49\linewidth}
\centering
\scalebox{0.6}{\includegraphics[width=\textwidth]{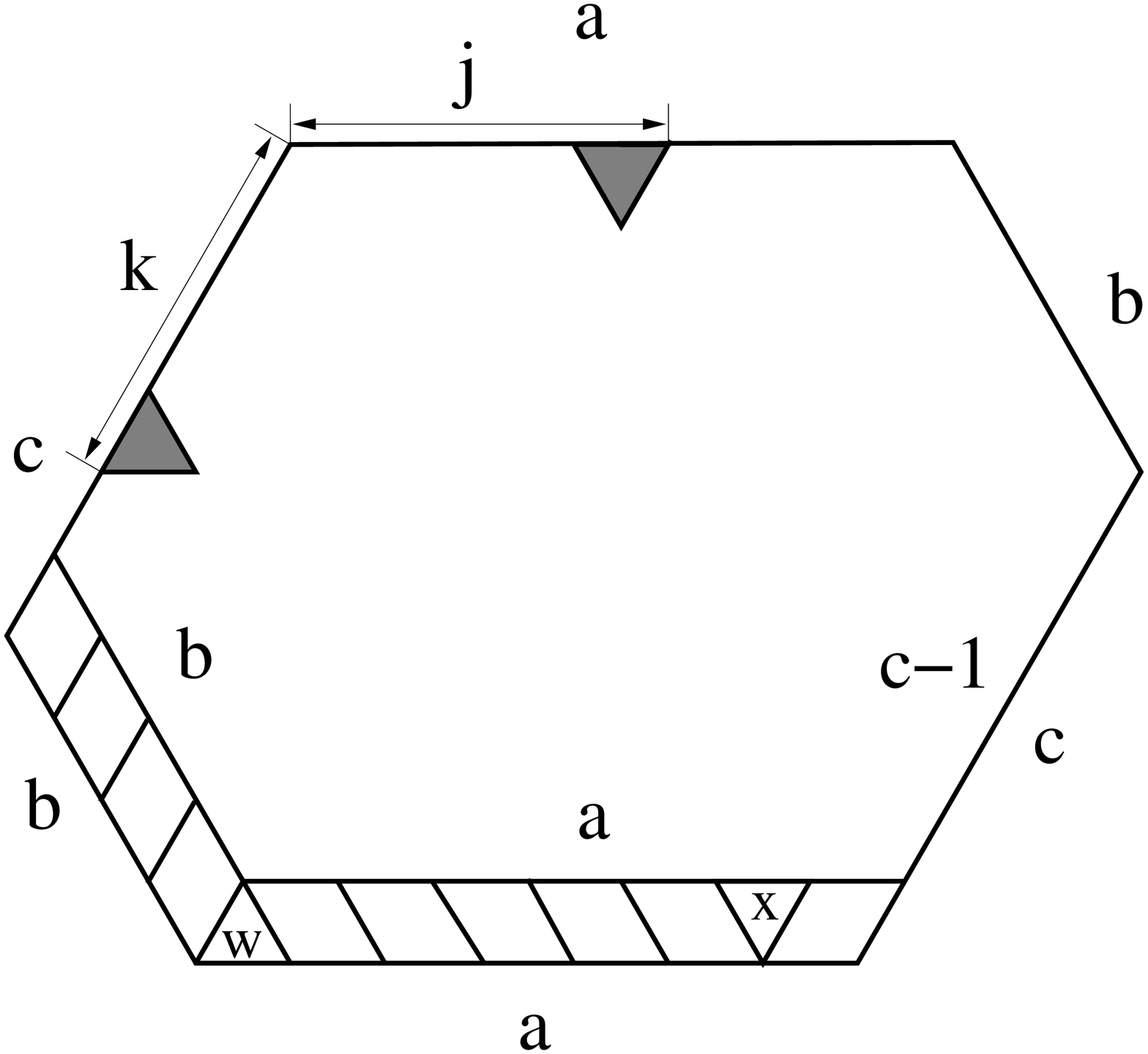}}
\end{minipage}
\caption{\label{adjacentkuo12_fig} $\M(G-\{w,x,y,z\})=\adj(a-1,b,c-1)_{j,k}$ and $\M(G-\{w,x\})=\adj(a,b,c-1)_{j,k}$.}
\end{figure}

\begin{figure}[ht]
\begin{minipage}[b]{0.49\linewidth}
\centering
\scalebox{0.6}{\includegraphics[width=\textwidth]{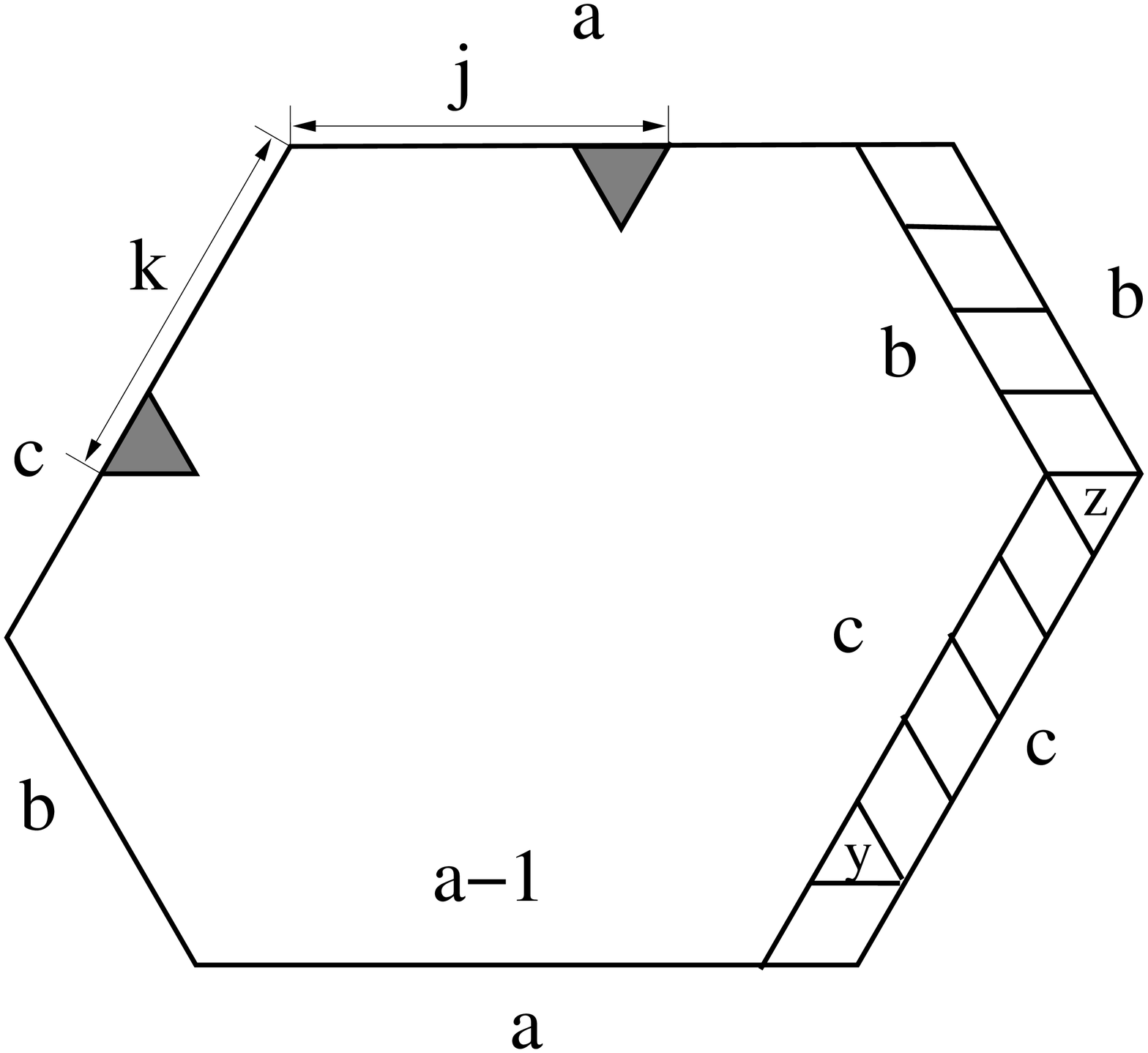}}
\end{minipage}
\begin{minipage}[b]{0.49\linewidth}
\centering
\scalebox{0.6}{\includegraphics[width=\textwidth]{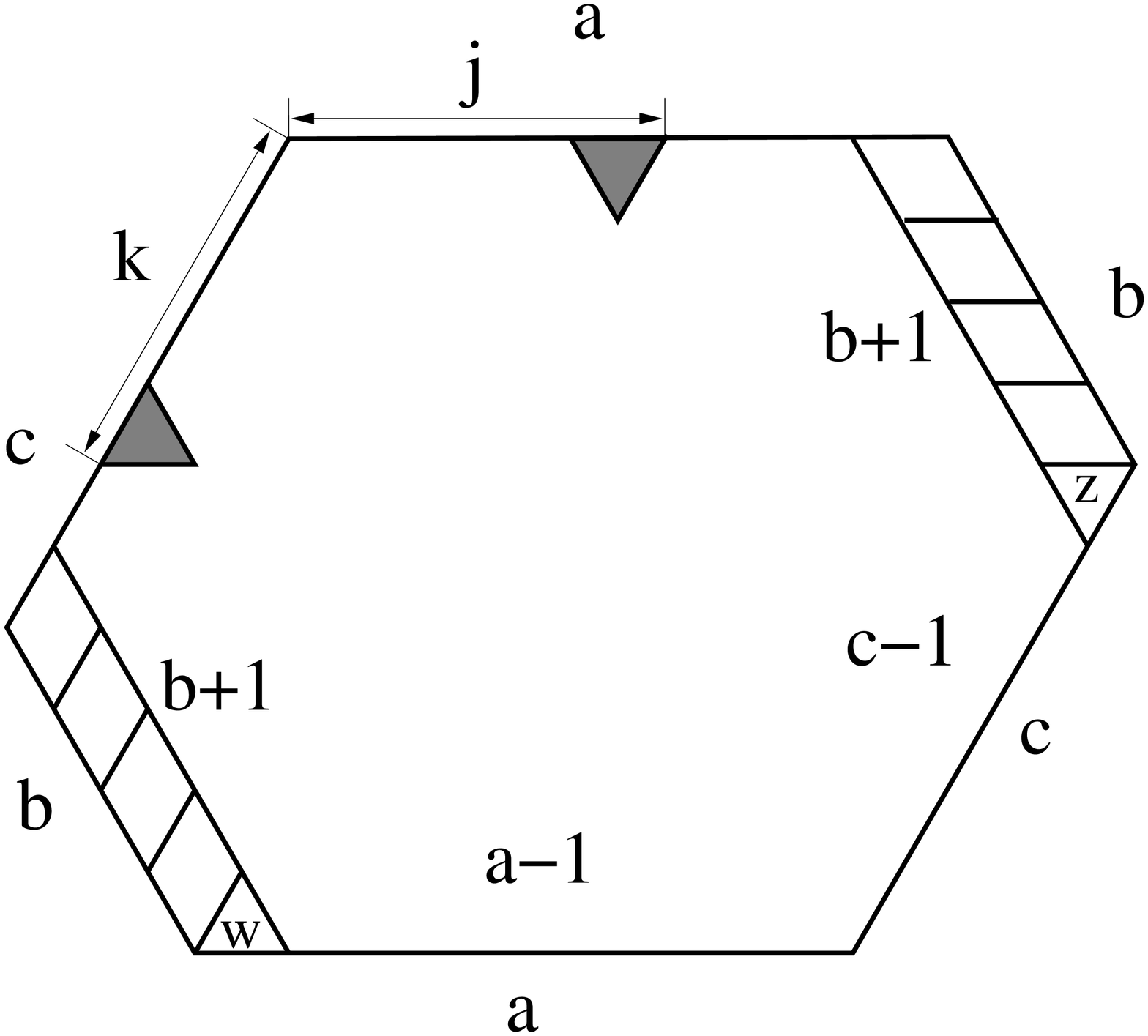}}
\end{minipage}
\caption{\label{adjacentkuo34_fig} $\M(G-\{y,z\})=\adj(a-1,b,c)_{j,k}$ and $\M(G-\{w,z\})=\adj(a-1,b+1,c-1)_{j,k}$.}
\end{figure}

\begin{figure}
\scalebox{0.22}{\includegraphics{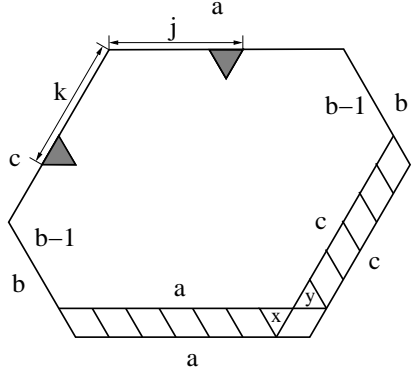}}
\caption{\label{adjacentkuo5_fig} $\M(G - \{x,y\})=\adj(a,b-1,c)_{j,k}$}
\end{figure}

Now it remains to show that the expression in the statement of the lemma fulfills \eqref{rec}. In doing so, we were assisted by 
Krattenthaler's  Mathematica package HYP for the manipulation of hypergeometric identities \cite{kratt}. We move all terms in 
\eqref{rec} to the left-hand side, plug in the expression for $\adj(a,b,c)_{j,k}$ and obtain an identity of the following structure:
\begin{multline}
\label{six}
 - \pFq{3}{2}{-a+j,-1+b,-c+k}{1-a-c, b}{1}  \pFq{3}{2}{1-a+j,1+b,1-c+k}{3-a-c, 2+b}{1}  \times \fpp_1\\
 - \pFq{3}{2}{-a+j,b,1-c+k}{2-a-c, 1+b}{1}  \pFq{3}{2}{1-a+j,b,-c+k}{2-a-c, 1+b}{1} \times  \fpp_2 \\
 + \pFq{3}{2}{-a+j,b,-c+k}{1-a-c, 1+b}{1}  \pFq{3}{2}{1-a+j,b,1-c+k}{3-a-c, 1+b}{1}  \times \fpp_3  = 0
\end{multline}
Here, $\fpp_i$, $i=1,2,3$, stands for certain fractions of products of Pochhammer functions. 

The six hypergeometric series in \eqref{six} differ from each other in every parameter by integer values of at most $2$. Our strategy is to apply contiguous relations in such a way that 
the resulting expression contains only one of these hypergeometric series -- multiple occurrences possible. In fact, this expression is then a polynomial of degree no greater than $2$ in this hypergeometric series. However, as it turns out, the three coefficients (which are sums of fractions of products of Pochhammer functions) of the polynomial vanish.

We start by applying ${\tt C55}[1-c+k,b,3-a-c]$ to 
$$
\pFq{3}{2}{1-a+j,b,1-c+k}{3-a-c, 1+b}{1}.
$$
The results is 
\begin{multline*}
\frac{b (2-a-k)}{(2-a-b-c)(-c+k)} \pFq{2}{1}{1-a+j,-c+k}{3-a-c}{1}  \\ + 
\frac{(2-a-c)(b+c-k)}{(-2+a+b+c)(-c+k)} \pFq{3}{2}{1-a+j,b,-c+k}{2-a-c,1+b}{1}.
\end{multline*}
We have a cancellation of an upper parameter with a lower parameter in the first hypergeometric series so that one 
$_{3} F_{2}$-series is converted into a $_{2} F_{1}$-series  and Chu-Vandermonde summation \eqref{chuvandermonde} can be applied (both $-1+a-j$ and $c-k$ are non-negative integers). The remaining $_{3} F_{2}$-series also appears elsewhere on the left-hand side of \eqref{six} and so we have reduced the number of different $_{3} F_{2}$-series from six to five. This will be the typical situation in our computation.

Next we apply ${\tt C55}[1-c+k,b,2-a-c]$ to 
$$
\pFq{3}{2}{-a+j,b,1-c+k}{2-a-c, 1+b}{1}.
$$
Again we have a cancellation so that one hypergeometric series is actually a $_{2} F_{1}$-series and Chu-Vandermonde can be applied. The other series is
$$
\pFq{3}{2}{-a+j,b,-c+k}{1-a-c,1+b}{1}
$$
and this is also a series appearing elsewhere in the expression.

Now we apply to ${\tt C54}[-a+j,b,1-a-c]$ to the two copies of
$$
\pFq{3}{2}{-a+j,b,-c+k}{1-a-c, 1+b}{z}
$$
in the expression.
This leads to
$$
\pFq{3}{2}{1-a+j,b,-c+k}{2-a-c,1+b}{1}
$$
as well as an $_{2} F_{1}$-series where Chu-Vandermonde summation is applicable.

After applying ${\tt C40}[-c+k,b]$ to 
$$
\pFq{3}{2}{-a+j,-1+b,-c+k}{1-a-c,b}{1},
$$
the two remaining series are 
$$
\pFq{3}{2}{1-a+j,1+b,1-c+k}{3-a-c, 2+b}{1}  \quad \text{and} \quad
\pFq{3}{2}{1-a+j,b,-c+k}{2-a-c,1+b}{1}. 
$$
We apply ${\tt C42}[1-a+j,2+b]$ to the 
first series and obtain an expression such that, after applying Chu-Vandermonde another time,
the second hypergeometric series is the only one appearing in our expression. The expression is then a polynomial of degree no greater than $2$ in this series. It is tedious but routine to check that the coefficients of the polynomial are in fact zero.  
\end{proof}

\begin{prop} 
\label{opposite_prop} Let $a,b,c,i,j$ be positive integers with $1 \le i, j \le a$.
The number of lozenge tilings of the hexagon $H_{a,b,c}$ with two dents in positions $i$ and $j$ along opposite sides of length $a$ (see Figure~\ref{opposite_fig}, left) is
\begin{multline*}
\prod_{k=0}^{a-2} \frac{(1+c+k)_b}{(1+k)_b} 
\pFq{4}{3}{1-i,1-j,1-c-j,1+a+b-j}{2-c-j,1+b-j,2+a-i-j}{1} \\ \times
\frac{(c)_{j-1} (1+b-j)_{i-1} (2+a-i-j)_{i+j-2}}{(1)_{i-1} (1)_{j-1} (1+a+c-i)_{i-1} (1+a+b-j)_{j-1}},
\end{multline*}
where the position of the first dent is counted from the common vertex of the respective side of length $a$ with the side of length $b$, while the position of the second dent is counted from the common vertex of the respective side of length $a$ with the side of length $c$. \end{prop}

\begin{figure}[ht]
\begin{minipage}[b]{0.49\linewidth}
\centering
\scalebox{0.6}{\includegraphics[width=\textwidth]{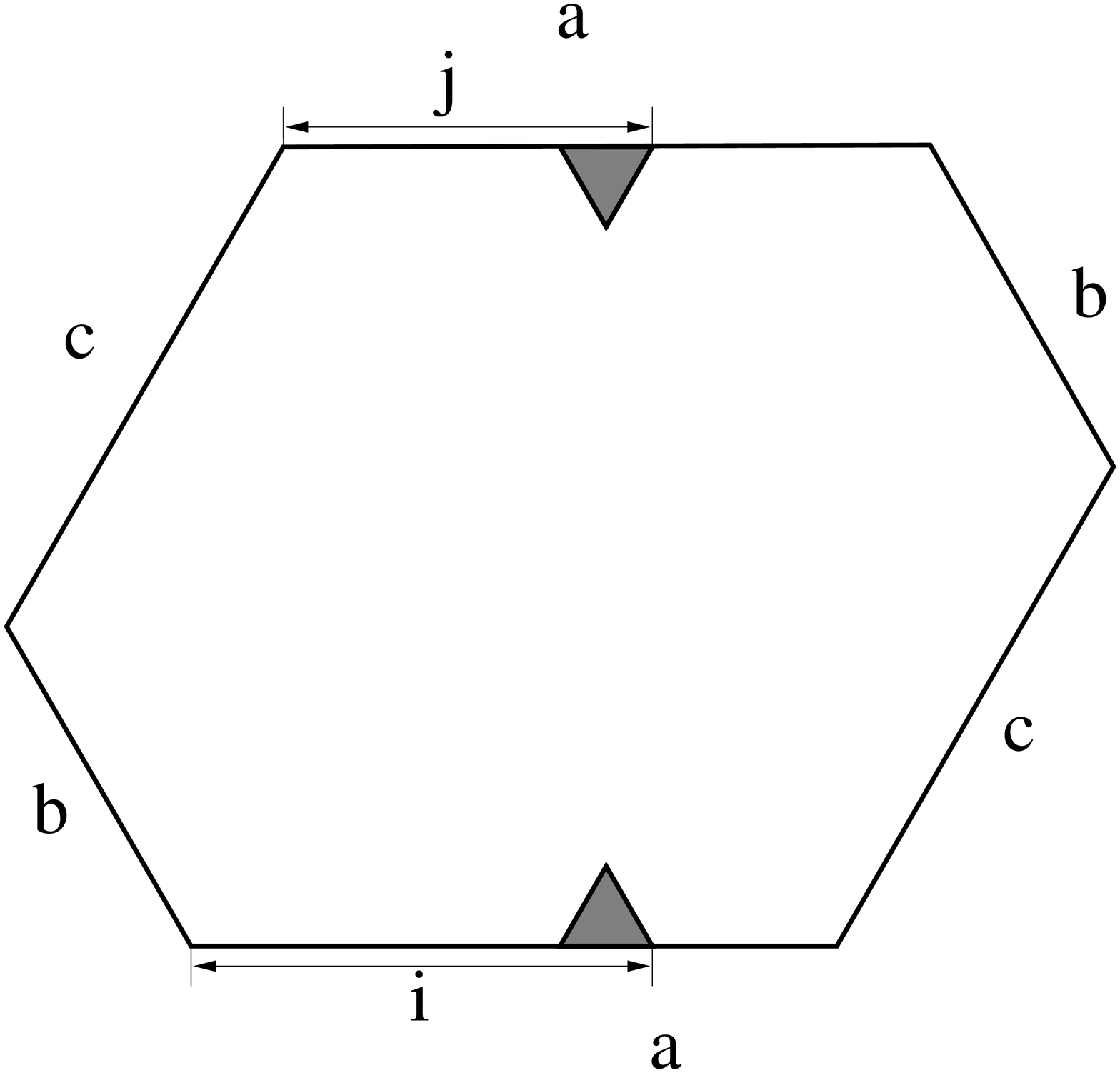}}
\end{minipage}
\begin{minipage}[b]{0.49\linewidth}
\centering
\scalebox{0.6}{\includegraphics[width=\textwidth]{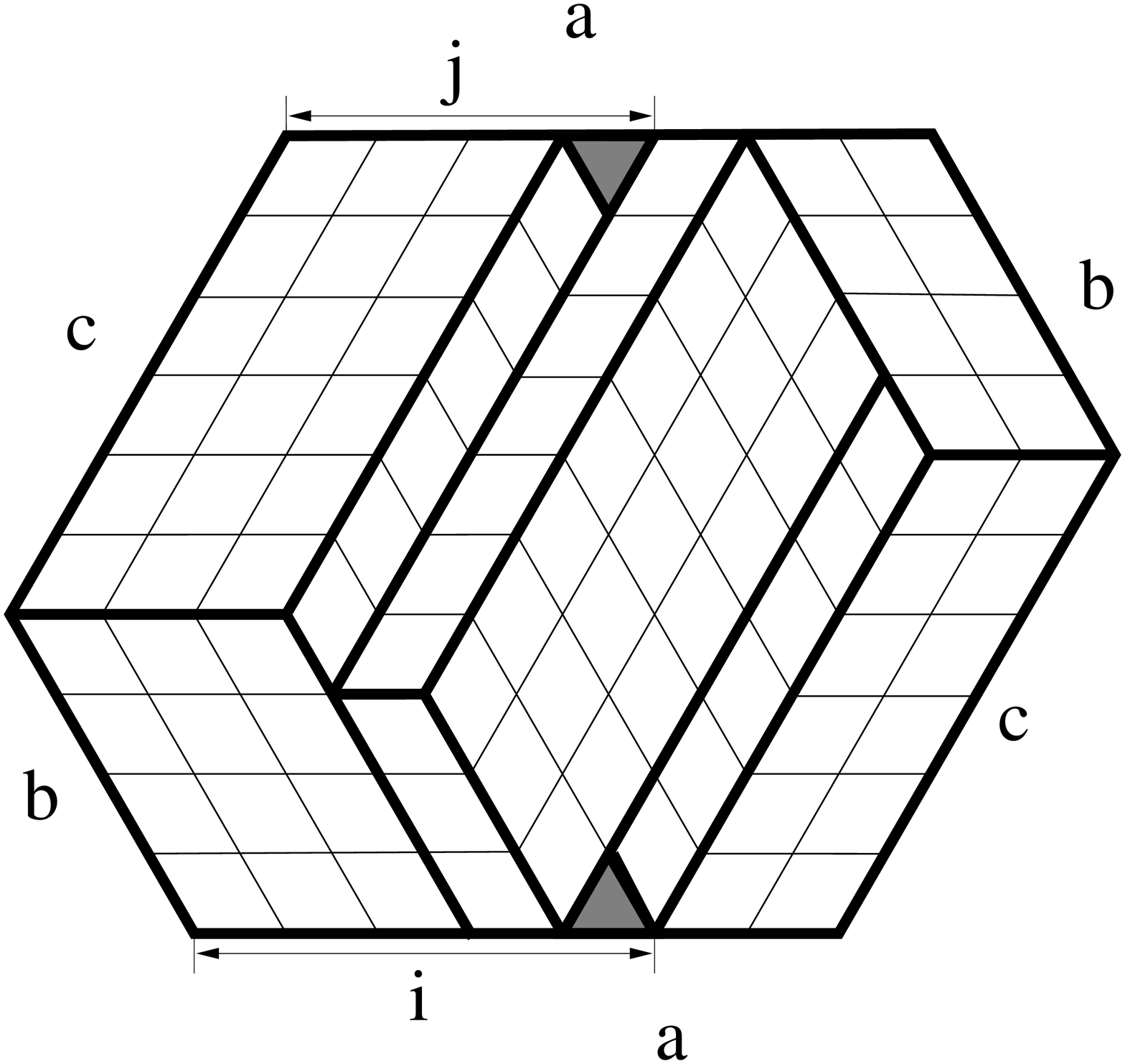}}
\end{minipage}
\caption{\label{opposite_fig} The region in Proposition~\ref{opposite_prop} and a canonical lozenge tiling.}
\end{figure}

Again the main ingredient for the proof of the proposition is Kuo's condensation. The following version will be applied.

\begin{theo} \cite[Theorem 2.3]{Kuo}  
\label{kuo2}
Let $G=(V_1,V_2,E)$ be a plane bipartite graph and $w,x,y,z$ vertices of $G$ that appear in cyclic order on a face of $G$. If $w,x \in V_1$ and $y,z \in V_2$ then
$$
\M(G) \M(G - \{w,x,y,z\}) =
\M(G - \{w,z\}) \M(G - \{x,y\}) - \M(G - \{w,y\}) 
\M(G - \{x,z\}).
$$
\end{theo}

Also in the proof of Proposition~\ref{opposite_prop} we need to apply some contiguous relations. Specifically, these are 
\begin{align*}
\pFq{r}{s}{x,y,(A)}{(B)}{z} &\stackrel{{\tt C30}[x,y]}{\rightarrow} \pFq{r}{s}{x-1,y+1,(A)}{(B)}{z} 
\\ & \qquad \qquad + (1-x+y) \frac{\prod_{i=1}^{r-2} A_i }{\prod_{i=1}^{s} B_i} \pFq{r}{s}{x,y+1(A+1)}{(B+1)}{z} \\
\pFq{r}{s}{w,(A)}{x,y,(B)}{z} &\stackrel{{\tt C57}[w,x,y]}{\rightarrow} \frac{(x-1)(y-w)}{(w-1)(y-x)} \pFq{r}{s}{w-1,(A)}{x-1,y,(B)}{z} \\
& \qquad \qquad + \frac{(x-w)(y-1)}{(w-1)(x-y)} \pFq{r}{s}{w-1,(A)}{x,y-1(B)}{z} 
\end{align*}
and ${\tt C55}[x,y]$ which was also used in the proof of Proposition~\ref{adjacent_prop} and introduced there. The role of the Chu-Vandermonde 
summation is now taken over by the Pfaff-Saalsch\"utz summation.
\begin{equation}
\label{pfaff}
\pFq{3}{2}{w,x,-n}{y}{1} \stackrel{{\tt S3201}}{\rightarrow} 
\frac{(y-w)_n (y-x)_n}{(y)_n (y-w-x)_n}
\end{equation}
Again $n$ has to be a non-negative integer.

\begin{proof}
We use induction with respect to $a$. The base cases of the induction are 
$a=1,2$. For our argument we also need to check that cases $i=1$, $j=1$, $i=a$ and $j=a$ individually. 

Since $a=1,2$ implies $i=1$ or $i=2$, we do not have to consider the cases $a=1,2$
once the other cases mentioned have been considered.
By the symmetry between $i$ and $j$, we do also not have to consider that cases  $j=1$ and $j=a$. Moreover, the case $i=a$ 
is equivalent to the case $i=1$ and so it suffices to consider the case $i=1$. This case 
is the special case $k=0$ in Proposition~\ref{gk_regions_prop} or a special case of Proposition~\ref{CLP}.

\begin{figure}
\scalebox{0.25}{\includegraphics{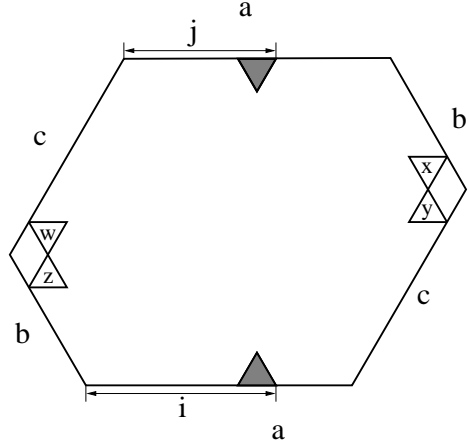}}
\caption{\label{oppositekuo_fig} The special vertices $w,x,y,z$ in Kuo's condensation in Proposition~\ref{opposite_prop}.}
\end{figure}

In the rest of the proof, we may assume $a \ge 3$ and 
$1 < i, j < a$ and let $\opp(a,b,c)_{i,j}$ denote the number of lozenge tilings of the region that is the subject of this proposition. Using Theorem~\ref{kuo2} with the vertices as indicated in Figure~\ref{oppositekuo_fig}, left, we obtain the following identity.
\begin{multline}
\label{rec2}
\opp(a,b,c)_{i,j} \opp(a-2,b,c)_{i-1,j-1} \\ = 
\opp(a-1,b,c)_{i-1,j-1} \opp(a-1,b,c)_{i,j} - 
\opp(a-1,b-1,c+1)_{i,j-1} \opp(a-1,b+1,c-1)_{i-1,j}
\end{multline}

In Figure~\ref{opposite_fig}, right, we indicate how to construct a lozenge tiling for any choice of parameters as described in the statement of the proposition. This implies $\opp(a-2,b,c)_{i-1,j-1} \not=0$ under our assumptions and so 
\eqref{rec2} provides a recursion for $\opp(a,b,c)_{i,j}$ with respect to $a$.

It remains to show that the expression in the statement of the lemma fulfills 
\eqref{rec2}. We move all terms in 
\eqref{rec2} to one side, plug in the expression for $\opp(a,b,c)_{j,k}$ and obtain an identity of the following structure:
\begin{multline}
\label{six2}
 - \pFq{4}{3}{1-i,1-j,1-c-j,a+b-j}{2-c-j,1+b-j,1+a-i-j}{1}   \pFq{4}{3}{2-i,2-j,2-c-j,1+a+b-j}{3-c-j,2+b-j,3+a-i-j}{1}  \fpp_1\\
 + \pFq{4}{3}{1-i,1-j,1-c-j,1+a+b-j}{2-c-j,1+b-j,2+a-i-j}{1}  \pFq{4}{3}{2-i,2-j,2-c-j,a+b-j}{3-c-j,2+b-j,2+a-i-j}{1}  \fpp_2 \\
 +\pFq{4}{3}{1-i,2-j,1-c-j,a+b-j}{2-c-j,1+b-j,2+a-i-j}{1} \pFq{4}{3}{2-i,1-j,2-c-j,1+a+b-j}{3-c-j,2+b-j,2+a-i-j}{1}  \fpp_3 \\ = 0
\end{multline}
Again, $\fpp_i$, $i=1,2,3$, stands for certain fractions of products of Pochhammer functions. 

We apply ${\tt C57}[1+a+b-j,2-c-j,2+a-i-j]$ to
$$
\pFq{4}{3}{1-i,1-j,1-c-j,1+a+b-j}{2-c-j,1+b-j,2+a-i-j}{1}
$$
and obtain 
\begin{multline*}
 \frac{(1-b-i)(1-c-j)}{(a+c-i) (a+b-j)} \pFq{3}{2}{1-i,1-j,a+b-j}{1+b-j,2+a-i-j}{1}  \\
 + \frac{(1-a-b-c)(1+a-i-j)}{(-a-c+i)(a+b-j)} 
 \pFq{4}{3}{1-i,1-j,1-c-j,a+b-j}{2-c-j,1+b-j,1+a-i-j}{1}.
\end{multline*}
Note that there is a cancellation in the first $_{4} F_{3}$-series and so we can apply Pfaff-Saalsch\"utz summation \eqref{pfaff}. Moreover observe that the  
$_{4} F_{3}$-series appears also elsewhere in the expression and so we have reduced the number of different $_{4} F_{3}$-series from six to five. Finding reductions of this type is our strategy to prove \eqref{six2}.

Next we apply ${\tt C57}[1+a+b-j,3-c-j,3+a-i-j]$ to 
$$
\pFq{4}{3}{2-i,2-j,2-c-j,1+a+b-j}{3-c-j,2+b-j,3+a-i-j}{1}
$$
and once again obtain an $_{3} F_{2}$-series to which we can apply Pfaff-Saalsch\"utz summation. The other series is
$$
\pFq{4}{3}{2-i,2-j,2-c-j,a+b-j}{3-c-j,2+b-j,2+a-i-j}{1}.
$$

Now we apply ${\tt C57}[2-j,2-c-j,2+a-i-j]$ to 
$$
\pFq{4}{3}{1-i,2-j,1-c-j,a+b-j}{2-c-j,1+b-j,2+a-i-j}{1}
$$
and obtain an expression containing 
$$
\pFq{4}{3}{1-i,1-j,1-c-j,a+b-j}{2-c-j,1+b-j,1+a-i-j}{1}
$$
as well as an $_{3} F_{2}$-series to which we can apply Pfaff-Saalsch\"utz summation.

We apply ${\tt C55}[2-j,a+b-j,3-c-j]$ to the two copies of 
$$
\pFq{4}{3}{2-i,2-j,2-c-j,a+b-j}{3-c-j,2+b-j,2+a-i-j}{1}.
$$
and obtain an expression containing only the following two $_{4} F_{3}$-series:
$$
\pFq{4}{3}{2-i,1-j,2-c-j,1+a+b-j}{3-c-j,2+b-j,2+a-i-j}{1} \quad \text{and} \quad 
\pFq{4}{3}{1-i,1-j,1-c-j,a+b-j}{2-c-j,1+b-j,1+a-i-j}{1}
$$
Finally we apply ${\tt C30}[1-j,1-c-j]$ to the second series and obtain an expression that is a polynomial in the first series of degree at most $2$. However, it turns out that the coefficients of this polynomial vanish.
\end{proof}

%\newpage

\section{Proof of the main results}
\label{proofmainresult}

%(if this section comes out too short, it can be combined with Section 2)

Our proof of Theorem~\ref{main_thm} is based on the first author's extension \cite{gk} of Kuo's graphical condensation method. For convenience we include it below.

A {\it weighted} graph is a graph with weights (that could be considered indeterminates) on its edges. For a weighted graph $G$, $\M(G)$ denotes the sum of the weights of the perfect matchings of $G$, where the weight of a perfect matching is taken to be the product of the weights of its constituent edges (note that if all edges have weight 1, this becomes simply the number of perfect matchings of the graph).

\begin{theo} \cite[Theorem 2.1]{gk}
\label{gk_thm}
Let $G$ be a planar graph with the vertices $\al_1,\dotsc,\al_{2k}$ appearing in that cyclic order on a face of $G$. Consider the skew-symmetric matrix $A=(a_{ij})_{1\leq i,j\leq 2k}$ with entries given by
\begin{equation}
%\label{dz_form}
a_{ij}:= \begin{cases} 
\phantom{-}\M(G\setminus\{\al_i,\al_j\}),  &\text{\rm if $i<j$}\\
-\M(G\setminus\{\al_i,\al_j\}),  &\text{\rm if $i>j$}.
\end{cases}
\end{equation}

%\begin{equation}
%a_{ij}:=\left\{\matrix
%\M(G\setminus\{a_i,a_j\}),\ \ \ \text{\rm if $i<j$},\\
%-\M(G\setminus\{a_i,a_j\}),\ \ \ \text{\rm if $i>j$}.
%\endmatrix\right.
%\end{equation}
Then we have that
\begin{equation}
\label{Pf_formula}
\M(G\setminus\{\al_1,\dotsc,\al_{2k}\})
=
\frac{\Pf(A)}{\left[\M(G)\right]^{k-1}}.
\end{equation}

\end{theo}

{\it Proof of Theorem~\ref{main_thm}.} Apply the Pfaffian formula provided by Theorem~\ref{gk_thm} to the planar dual graph of the region $\bar{H}^k_{a,b,c}$, and the vertices $\de_1,\dotsc,\de_{2n+2k}$ (recall that the latter are a listing in cyclic order of the vertices of the dual graph corresponding to the unit triangles in the set $\{\al_1,\dotsc,\al_{n+k}\}\cup\{\be_1,\dotsc,\be_n\}\cup\{\ce_1,\dotsc,\ce_k\}$; see Figure~\ref{fbb} and the footnote at the end of Section~2). 

Then the left hand side of equation (\ref{Pf_formula}) becomes precisely the left hand side of equation~(\ref{pfaff_expr}), and the right hand side of (\ref{Pf_formula}) becomes the expression on the right in (\ref{pfaff_expr}). To complete the proof of the theorem, we need to verify that the quantities on the right hand side of (\ref{pfaff_expr}) are given by explicit formulas as described by $(i)$--$(iv)$ in the statement of Theorem~\ref{main_thm}.

\begin{figure}[h]
\centerline{
\hfill
{\includegraphics[width=0.45\textwidth]{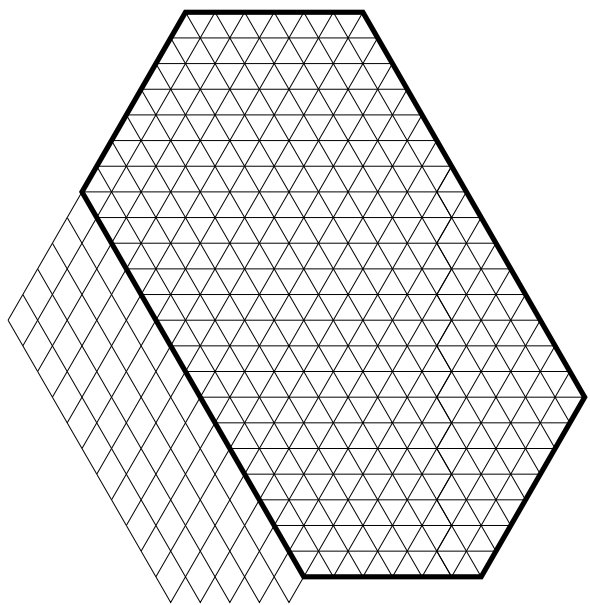}}
\hfill
%{\includegraphics[width=0.45\textwidth]{hex_2notches_p.eps}}
%\hfill
}
%\vskip-0.1in
\caption{\label{big_reg_forced} Removing the forced lozenges in $\bar{H}^5_{6,10,7}$.}
%\vskip-0.15in
\end{figure}

Statement $(i)$ readily follows, noting that $k$ strips of lozenges along the southwestern side of $\bar{H}^k_{a,b,c}$ are forced to be part of all of its tilings (see Figure~\ref{big_reg_forced}). Upon their removal, one ends up with a centrally symmetric hexagon, whose number of tilings is given by MacMahon's formula~(\ref{MacM_eq}).

\begin{figure}[h]
\centerline{
\hfill
{\includegraphics[width=0.45\textwidth]{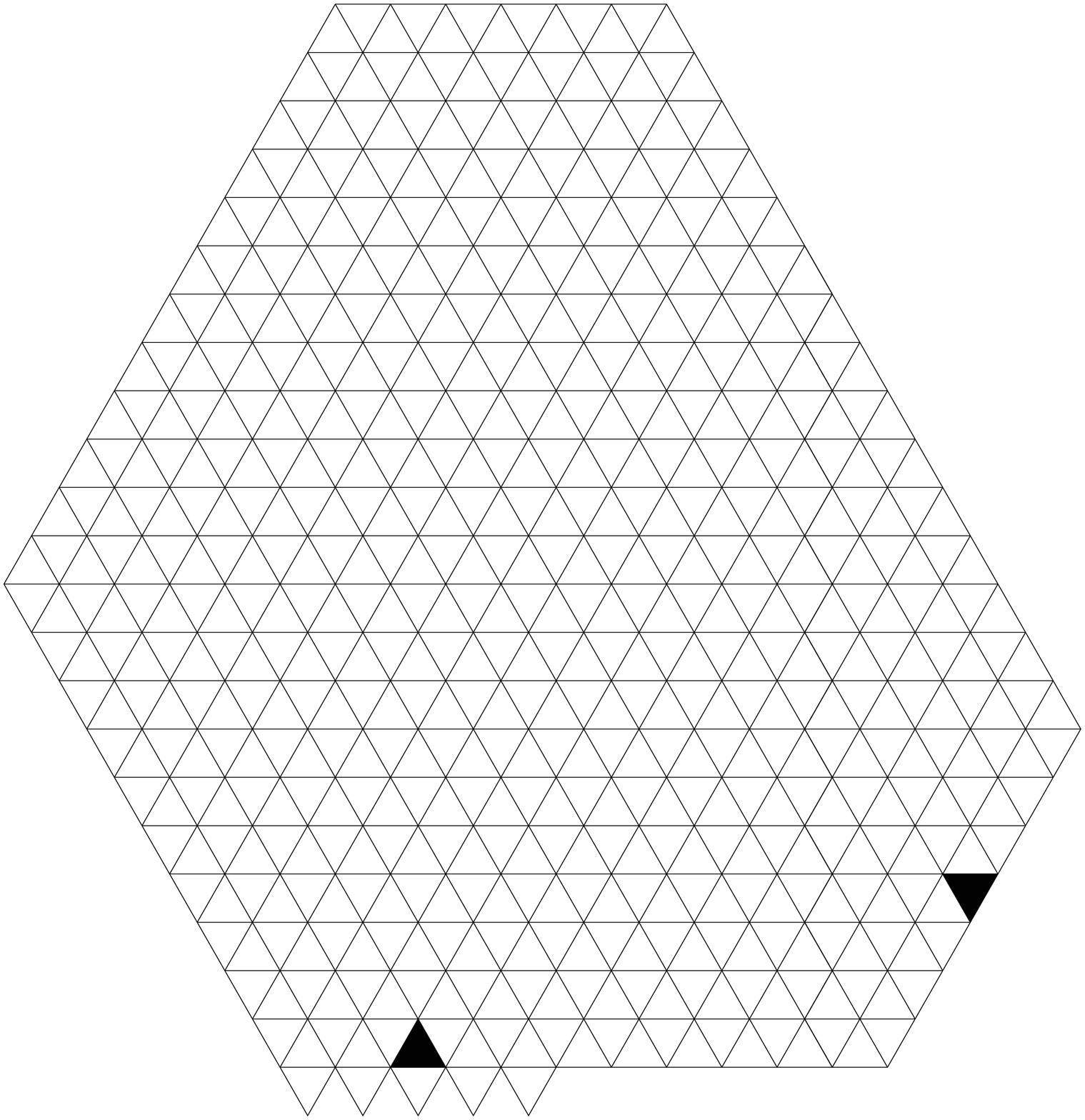}}
\hfill
{\includegraphics[width=0.45\textwidth]{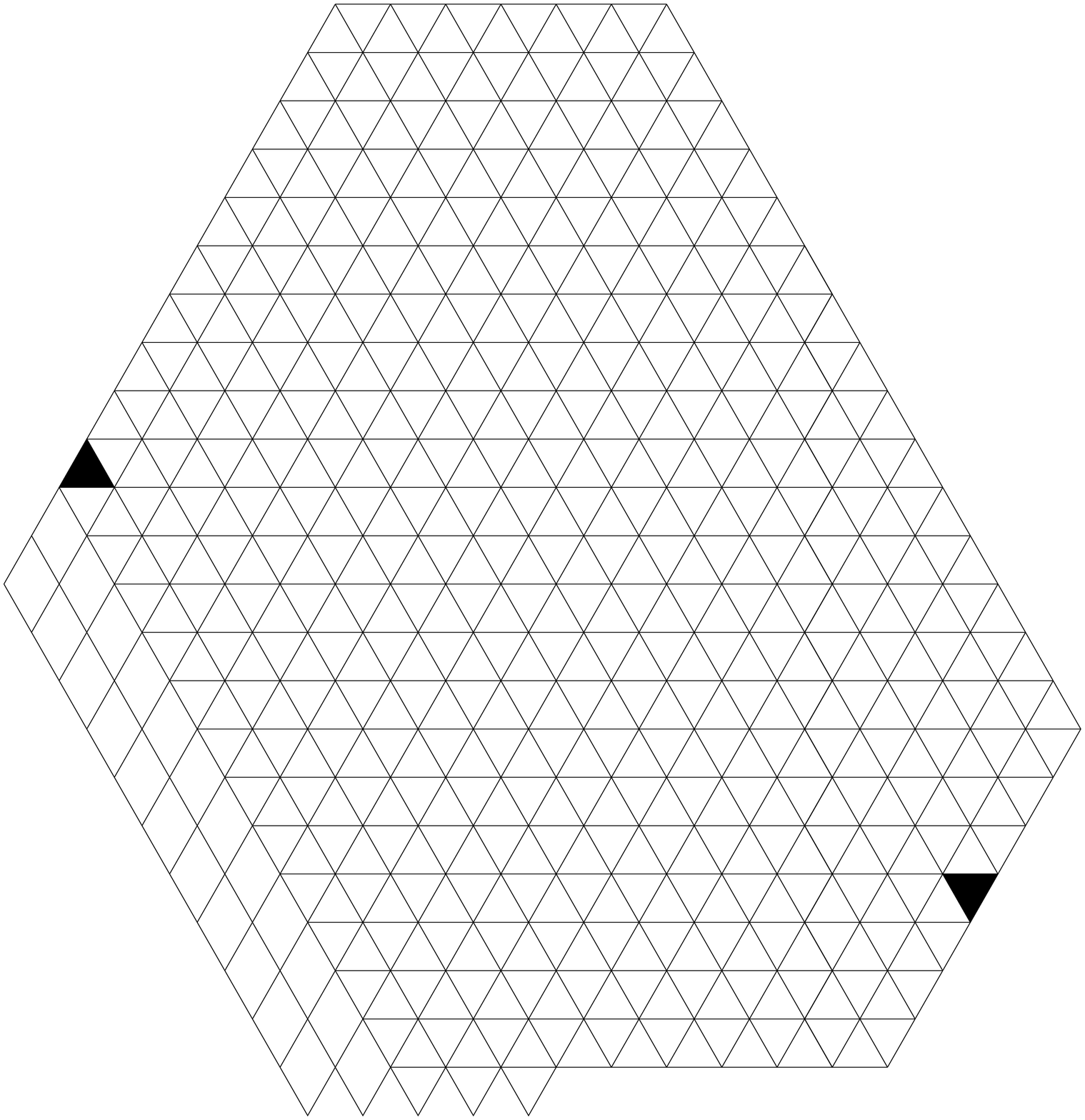}}
\hfill
}
%\vskip-0.1in
\caption{\label{rem_ab_is0} Two choices of an $\al$-dent and a $\be$-dent in $\bar{H}^5_{6,10,7}$ that lead to regions with no tilings.}
%\vskip-0.15in
\end{figure}

The two situations in the first part of statement $(ii)$ are illustrated in Figure~\ref{rem_ab_is0}. If $\al_i$ shares an edge with some $\ce_\nu$ (see the picture on the left in Figure~\ref{rem_ab_is0}), then $\ce_\nu$ cannot be covered by any lozenge in the region $\bar{H}^k_{a,b,c}\setminus\{\al_i,\be_j\}$, and hence $\M(\bar{H}^k_{a,b,c}\setminus\{\al_i,\be_j\})=0$. Similarly, if $\al_i$ is on the northwestern side, at a distance at most $k$ from the western corner (this situation is illustrated on the right in Figure~\ref{rem_ab_is0}), then the strips of forced lozenges along the southwestern side interfere with $\al_i$, and again there is no tiling.

Suppose therefore that $\al_i$ is in neither of the situations described in the first part of statement~$(ii)$. Then, due to the unit triangles $\ce_1,\dotsc,\ce_k$ on the bottom, there are $k$ strips of forced lozenges along the southwestern side of $\bar{H}^k_{a,b,c}$, as shown in Figure~\ref{big_reg_forced}, and $\al_i$ and $\be_i$ are dents on the boundary of the centrally symmetric hexagon left over after removing these forced lozenges. Since these two dents are unit triangles pointing in opposite directions, they must be either on adjacent or on opposite sides of the leftover hexagon, and statement $(ii)$ follows.

\begin{figure}[h]
\centerline{
\hfill
{\includegraphics[width=0.45\textwidth]{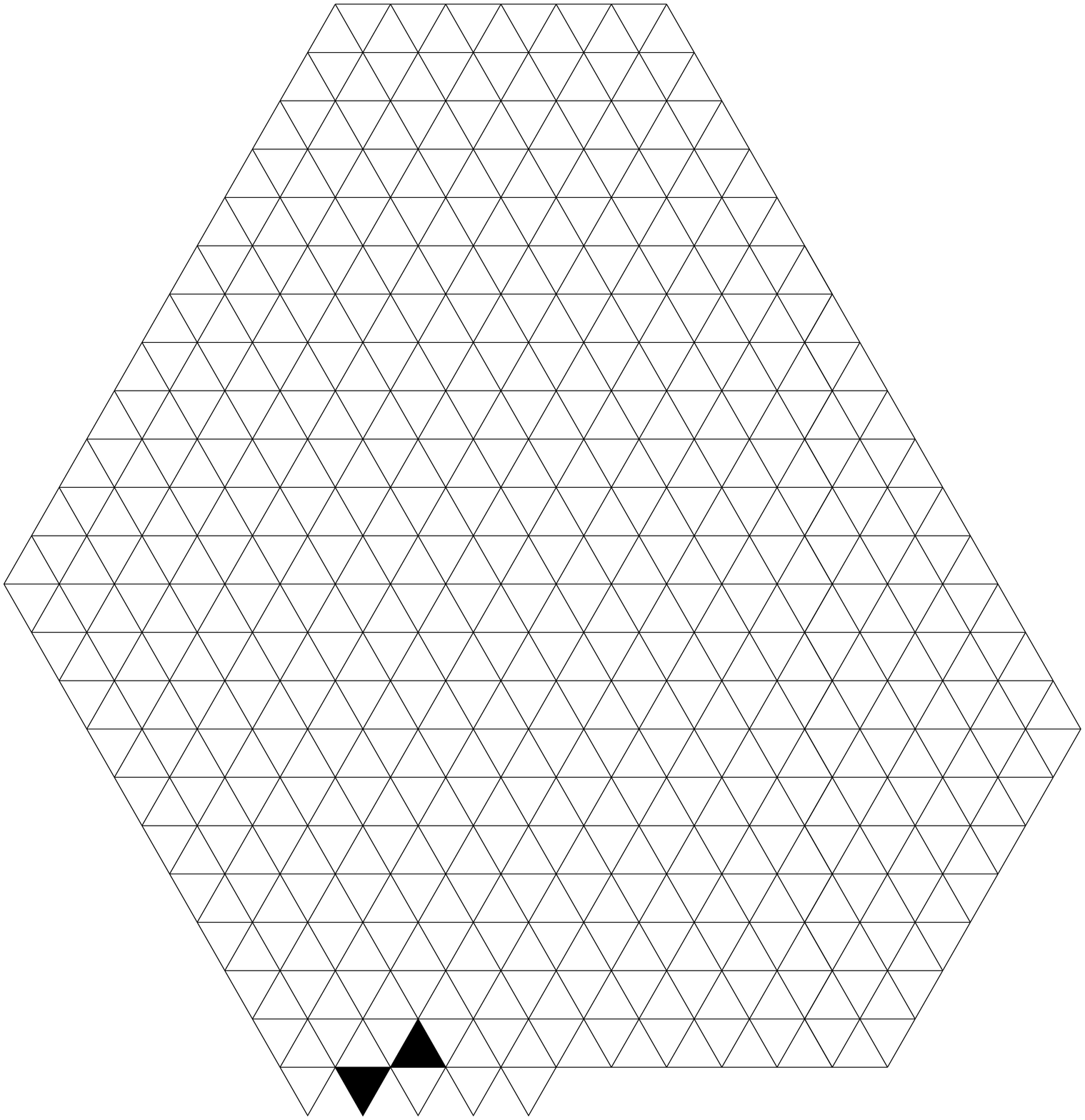}}
\hfill
{\includegraphics[width=0.45\textwidth]{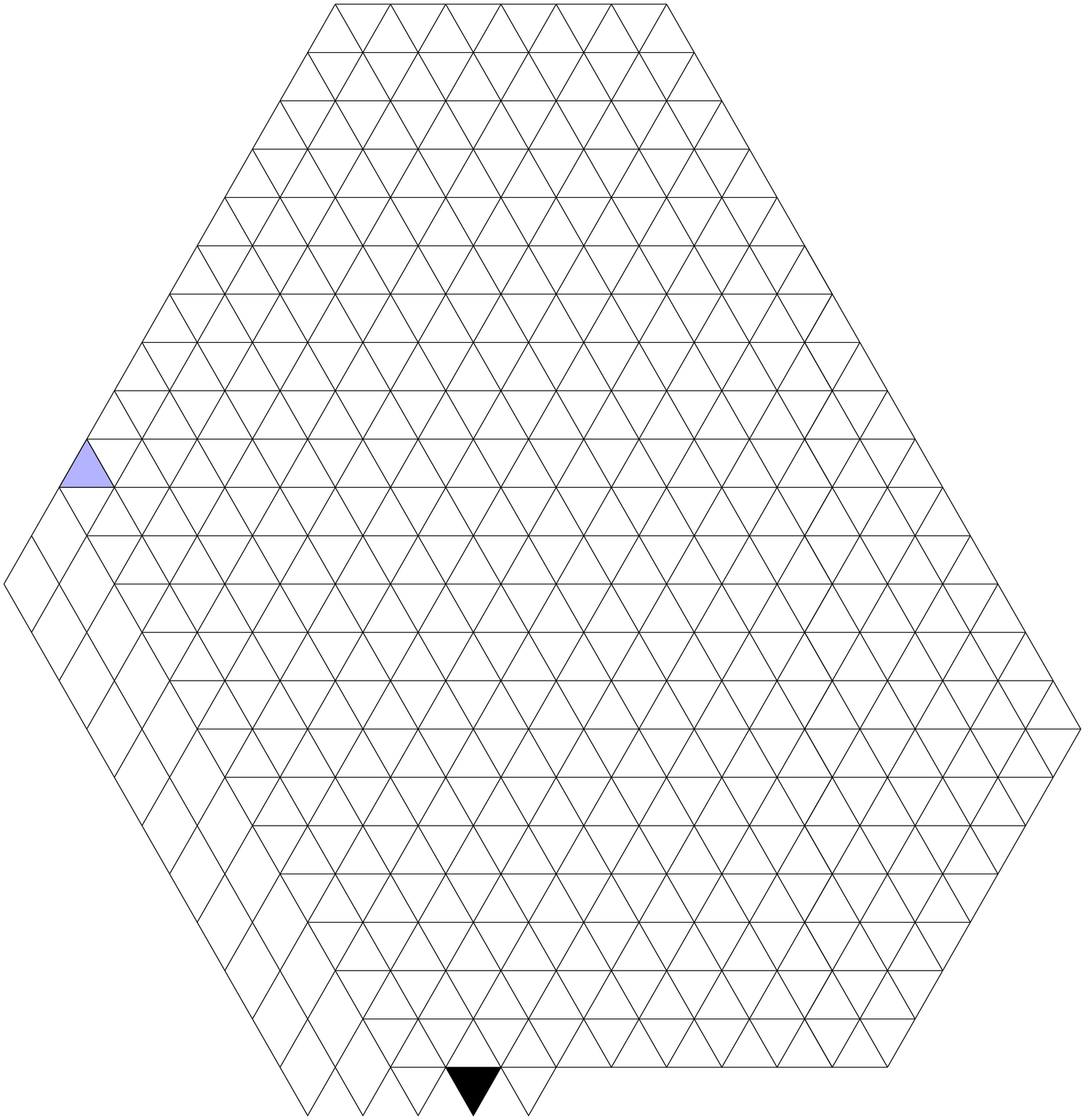}}
\hfill
}
%\vskip-0.1in
\caption{\label{rem_ac_is0} Two choices of an $\al$-dent and a $\ce$-dent in $\bar{H}^5_{6,10,7}$ that lead to regions with no tilings.}
%\vskip-0.15in
\end{figure}

The two situations described in the first part of statement $(iii)$ are illustrated in Figure~\ref{rem_ac_is0}. The resulting regions have no tilings for the same reasons as in the case of removing an $\al_i$ and a $\be_j$ discussed above.

\begin{figure}[h]
\centerline{
\hfill
{\includegraphics[width=0.33\textwidth]{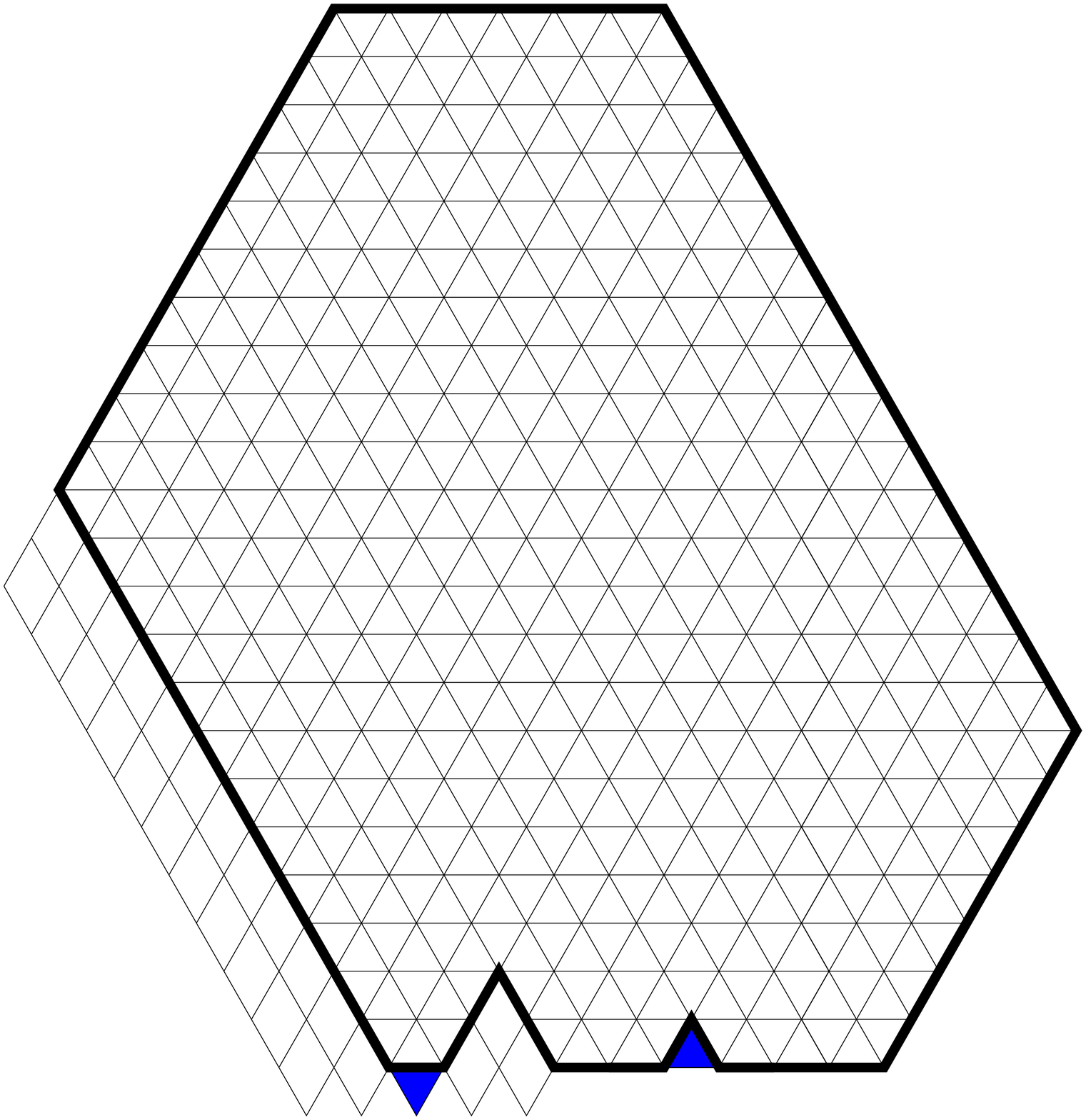}}
\hfill
{\includegraphics[width=0.33\textwidth]{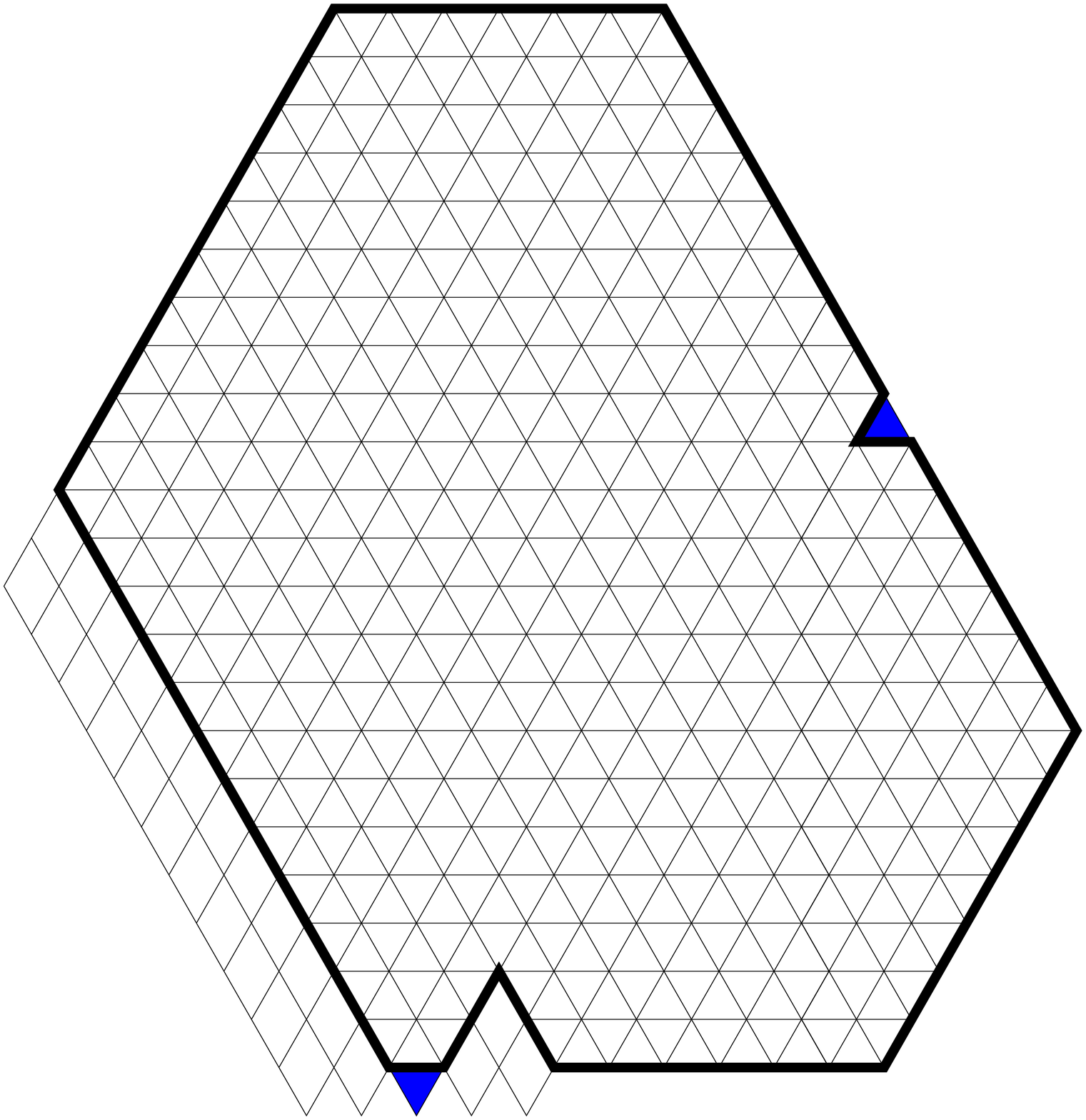}}
\hfill
{\includegraphics[width=0.33\textwidth]{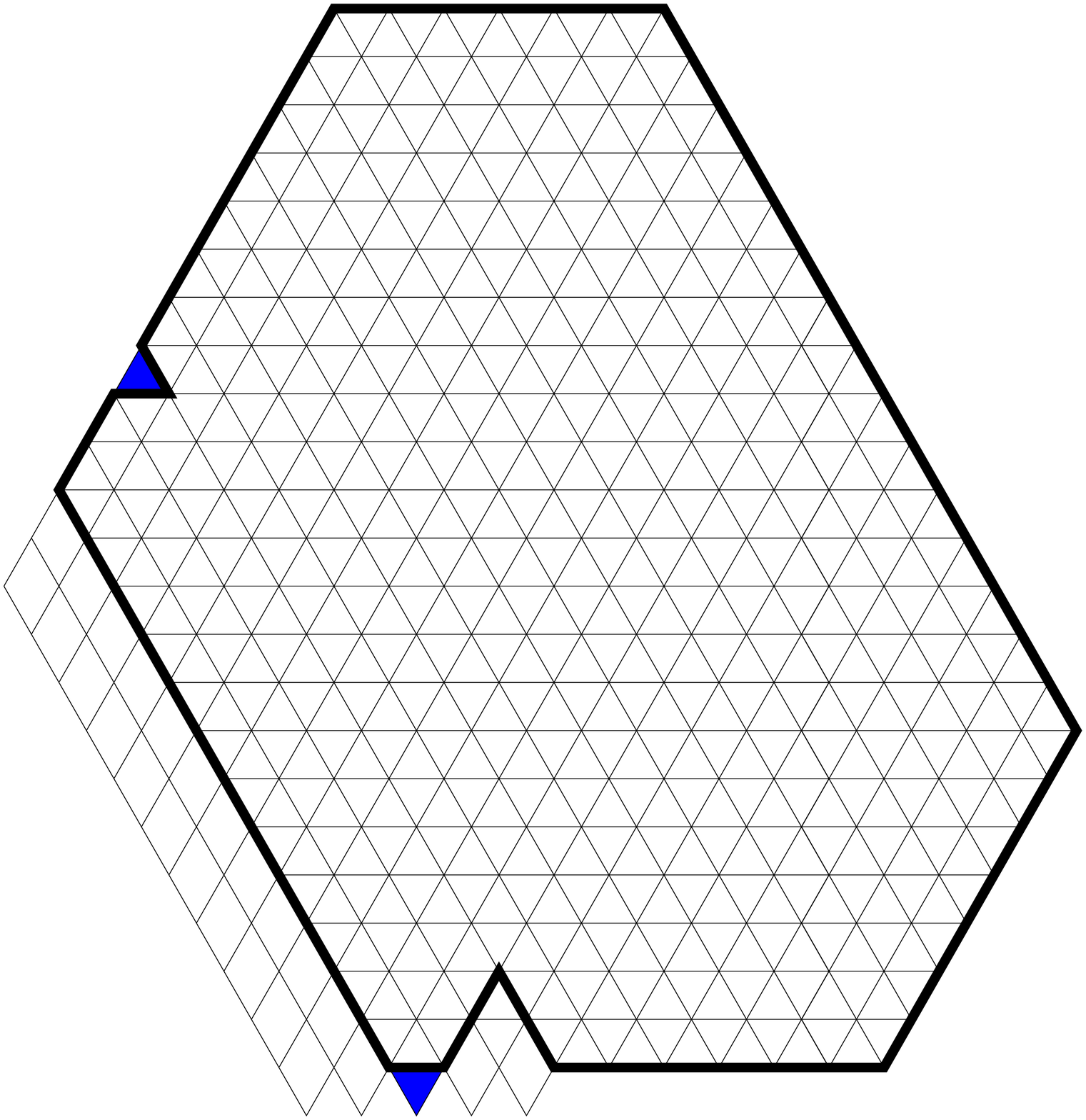}}
\hfill
}
%\vskip-0.1in
\caption{\label{rem_ac_not0} The three types of choices of an $\al$-dent and a $\ce$-dent in $\bar{H}^5_{6,10,7}$ that lead to regions that have tilings.}
%\vskip-0.15in
\end{figure}

If none of them applies, then the situation is one of the three described in Figure~\ref{rem_ac_not0}. In the first situation the region obtained after removing the forced lozenges is of the type covered by Proposition~\ref{CLP} (indeed, a dent of side $s$ is readily seen to be equivalent with a run of $s$ consecutive unit dents), while in the remaining two the resulting regions are precisely of the two kinds addressed by Proposition~\ref{gk_regions_prop}. This completes the verification of statement $(iii)$.

Statement $(iv)$ readily follows from the fact that a necessary condition for a region on the triangular lattice to have a lozenge tiling is to have the same number of up-pointing and down-pointing unit triangles. This completes the proof of the theorem. \epf

\begin{figure}[h]
\centerline{
\hfill
{\includegraphics[width=0.45\textwidth]{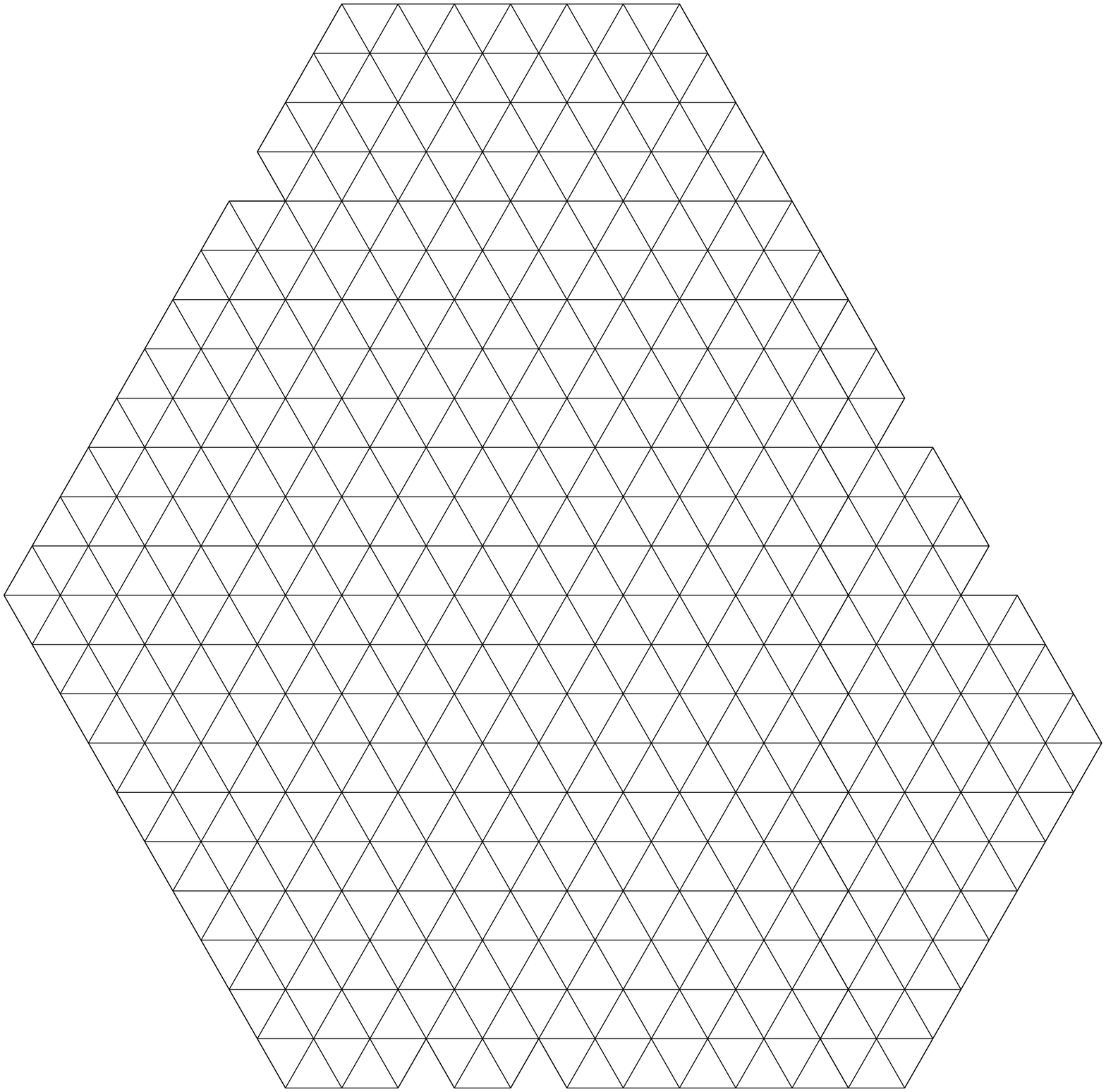}}
\hfill
{\includegraphics[width=0.45\textwidth]{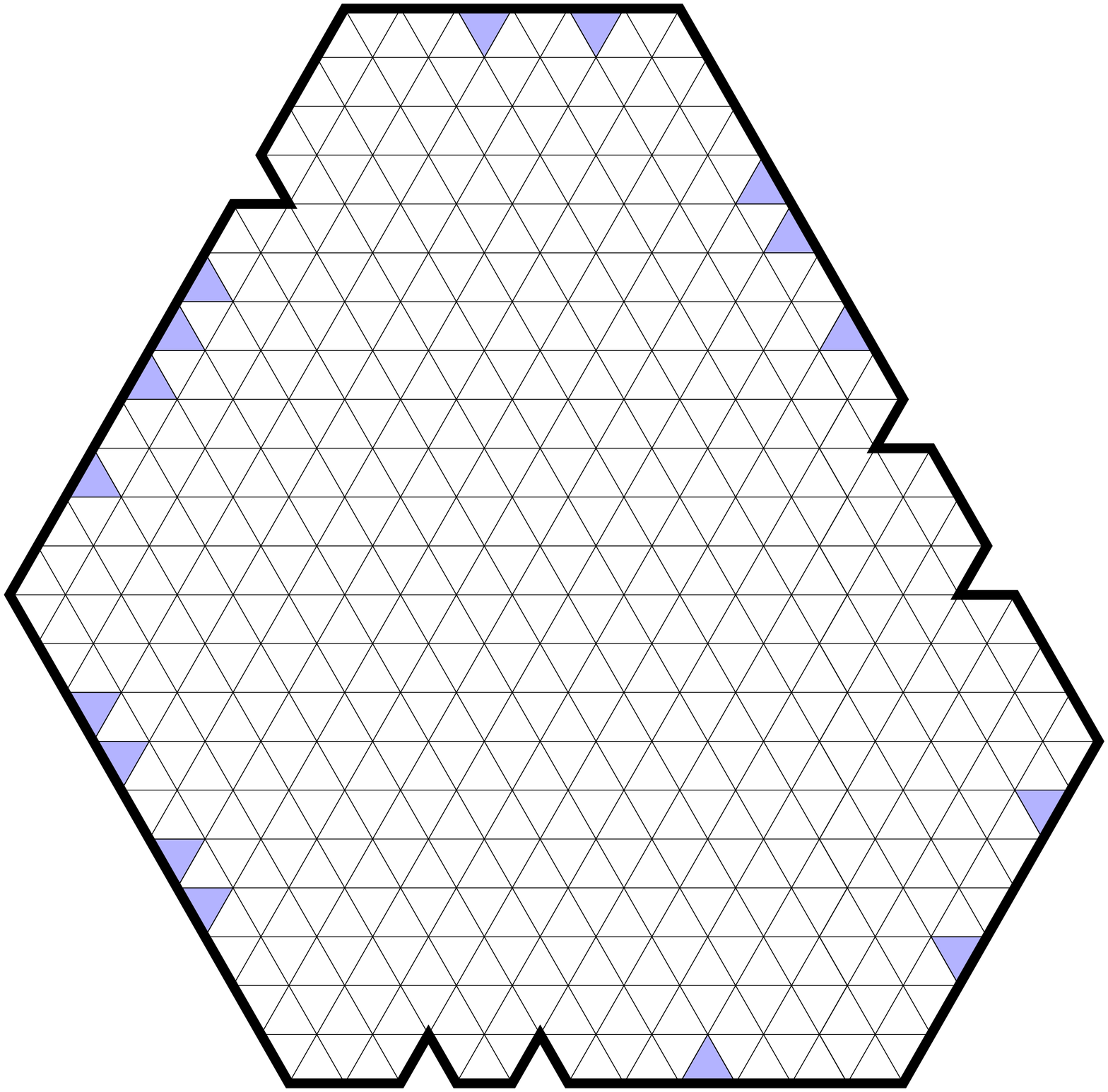}}
\hfill
%{\includegraphics[width=0.195\textwidth]{two4fs.eps}}
%\hfill
}
%\vskip-0.1in
\caption{\label{a_dent_bdry} A region whose boundary accounts for some of the $\al$-dents and a choice of unit triangles along its boundary.}
%\vskip-0.15in
\end{figure}

%\medskip
{\it Proof of Theorem~\ref{main_thm_gen}.} Let $D$ be the region obtained from $H^k_{a,b,c}$ by removing $k$ of the unit triangles $\al_1,\dotsc,\al_{n+k}$ (this region is illustrated on the left in Figure~\ref{a_dent_bdry}). Apply Theorem~\ref{gk_thm} to the planar dual graph of $D$, with the removed unit triangles chosen to be the vertices corresponding to the $n$ $\al_i$'s inside $D$ and to $\be_1,\dotsc,\be_n$. Then the left hand side of equation (\ref{Pf_formula}) becomes precisely the number of tilings we need, and the right hand side of (\ref{Pf_formula}) becomes the Pfaffian of a $2n\times2n$ matrix whose entries are of the form $\M(D\setminus\{\al_i,\be_j\})$, where $\al_i$ is not one of the unit triangles that were removed from $H^k_{a,b,c}$ to obtain $D$. However, $D\setminus\{\al_i,\be_j\}$ is a dented hexagon with all dents confined to four of its sides (the dents of type $a$ can only occur along the northwestern, northeastern, and southern sides of the hexagon, and there is a single dent of type $\be$). Therefore Theorem~\ref{main_thm} applies, and 
%since $D\setminus\{a_i,b_j\}$ is a dented hexagon with $k+2$ dents, 
it provides an expression for $\M(D\setminus\{\al_i,\be_j\})$ as the Pfaffian of a $(2k+2)\times(2k+2)$ matrix of the type described in the statement of Theorem~\ref{main_thm}. \epf

%\medskip
{\it Proof of Theorem~\ref{kis0_thm}.} Apply Theorem~\ref{gk_thm} to the planar dual graph of $H_{a,b,c}$, with the removed unit triangles chosen to correspond to $\al_1,\dotsc,\al_n,\be_1,\dotsc,\be_n$. Then the right hand side of (\ref{Pf_formula}) becomes precisely the expression on the right hand side of (\ref{kis0_expr}).  If $\de_i$ and $\de_j$ are of the same type, $H_{a,b,c}\setminus\{\de_i,\de_j\}$ does not have the same number of up-pointing and down-pointing unit triangles, and $\M(H_{a,b,c}\setminus\{\de_i,\de_j\})=0$. To complete the proof, note that $H_{a,b,c}\setminus\{\al_i,\be_j\}$ is either a hexagon with two dents on adjacent sides, or  a hexagon with two dents on opposite sides, and hence its number of tilings is given by Proposition~\ref{adjacent_prop} or Proposition~\ref{opposite_prop}, respectively. \epf

%\newpage

\section{Concluding remarks and an open problem}

In this paper we presented Pfaffian expressions for hexagons with arbitrary dents along the boundary. If the dents are confined to five sides of the hexagon, or if there is the same number of up-pointing and down-pointing dents, the entries in our Pfaffians have explicit forms, as either products of linear factors or single sums of products of linear factors. The expression for the general case is a nested Pfaffian.
It would be interesting to find a Pfaffian expression with entries given explicitly in the general case.

%* Find a Pfaffian expression for general case with explicit entries 

\end{document}